\documentclass{amsart}

\usepackage{amsmath,amssymb,amsthm,graphicx}
\usepackage{graphicx,tikz}

\newtheorem*{proposition}{Proposition}

\newtheorem{lemma}{Lemma}

\theoremstyle{definition}

\theoremstyle{remark}

\begin{document}

\title[]{Finding Structure in Sequences of real \\numbers via Graph Theory: a problem list}
\thanks{S.S. is supported by the NSF (DMS-2123224) and the Alfred P. Sloan Foundation. This work was part of a Washington Experimental Mathematics Laboratory (WXML) project in Fall 2020. }

\author[]{Dana G. Korssjoen \and Biyao Li \and Stefan Steinerberger \and \\Raghavendra Tripathi \and Ruimin Zhang}
\address{Department of Mathematics, University of Washington, Seattle, WA 98203, USA}

\email{danagk@uw.edu}
\email{biyao73@uw.edu}
\email{steinerb@uw.edu}
\email{raghavt@uw.edu}
\email{ruiminz@uw.edu}

\begin{abstract} We investigate a method of generating a graph $G=(V,E)$ out of an ordered list of $n$ distinct real numbers $a_1, \dots, a_n$. These graphs can be used to test for the presence of combinatorial structure in the sequence. We describe sequences exhibiting intricate hidden structure that was discovered this way. Our list includes sequences of Deutsch, Erd\H{o}s, Freud \& Hegyvari, Recaman, Quet, Zabolotskiy and Zizka. Since our observations are mostly empirical, each sequence in the list is an open problem.
\end{abstract}
\maketitle

\vspace{-13pt}

\tableofcontents

\section{Outline}
\subsection{Introduction.} The purpose of this paper is three-fold:
\begin{enumerate}
\item First we explain a simple way in which one can turn any list of $n$ distinct real numbers $a_1, \dots, a_n$ into a 4-regular graph $G=(V,E)$.  This particular method has been proposed by the third author \cite{steinerberger} as a way of seeing whether these numbers could be the realization of i.i.d. random variables: if the $a_i$ are random, then the graph behaves like an expander graph.

\item This paper makes explicit use of the contrapositive statement: if the arising graph is highly structured, then this necessitates the presence of some rigorous structure in $a_1, \dots, a_n$. In particular, the regularity of the graph can be used as a way to detect a certain type of combinatorial structure in the sequence. We argue that the first nontrivial eigenvalue of the Graph Laplacian can be used as a quantitative way of detecting structure.

\item We then apply this philosophy to a large number of different sequences and we find the presence of fascinating combinatorial structures in many of them. This is the core part of the paper.  Most of these observations will be purely empirical -- thus, we effectively list a large number of open problems: prove that said sequence indeed exhibits the observed type of structure.\end{enumerate}

\subsection{A Type of Graph} 
\label{subsection:Graph}
 Let $a_1, \dots, a_n$ be a set of $n$ distinct, real numbers. We will associate to them a unique 4-regular graph as follows: we connect $a_{i}$ to $a_{i+1}$ (cyclically, so $a_n$ also gets connected to $a_1$). Then we order them in increasing size $a_{\pi(1)} < a_{\pi(2)} < \dots < a_{\pi(n)}$ and connect $a_{\pi(i)}$ to $a_{\pi(i+1)}$ and conclude by also connecting to $a_{\pi(n)}$ to $a_{\pi(1)}$. This results in a 4-regular graph.

\begin{center}
\begin{figure}[h!]
\begin{tikzpicture}
\node at (0,0) {\includegraphics[width=0.5\textwidth]{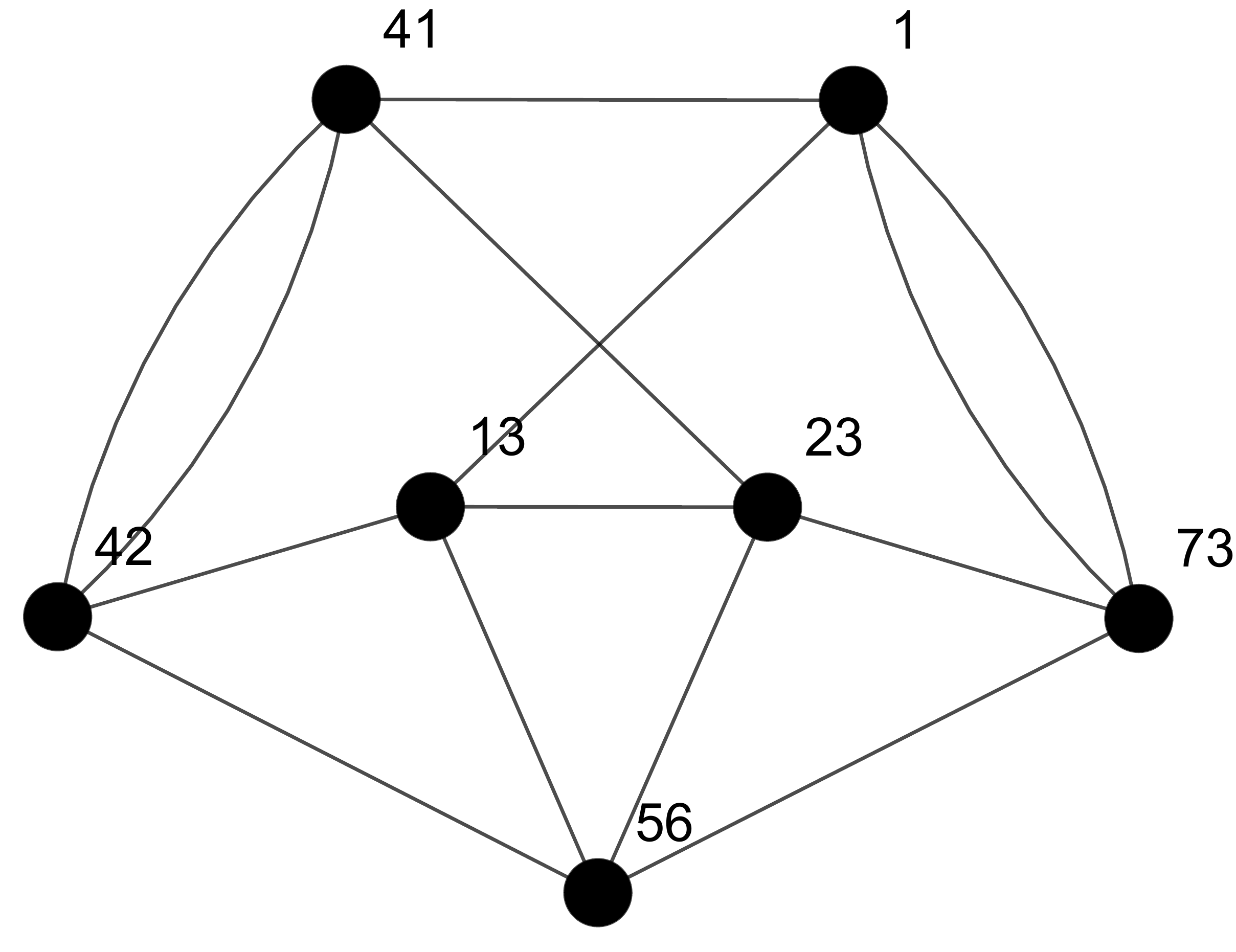}};
\end{tikzpicture}
\caption{The Graph for the first few digits of $\sqrt{2}$ in the order in which they appear: $1, 41, 42, 13, 56, 23, 73$.}
\label{fig:exa}
\end{figure}
\end{center}

We illustrate this with a simple example (see Fig. \ref{fig:exa}): when given the first few digits of $\sqrt{2}$ as $1, 41, 42, 13, 56, 23, 73$, we connect each element to its subsequent and prior element (i.e. 1 is connected to both 41 and 73, 56 is connected to 23 and 13). We then order the elements in increasing size, this leads to $1,13,23,41,42,56,73$ and once more connect each element to the subsequent and the prior element. We see that this can lead to multiple edges between two vertices (since 41,42 appear both in that order but also in the ordering by size). We also observe that as soon as $n \geq 3$, there are at most two edges between any pair of vertices. Moreover, the graph does not have any loops and is always 4-regular by construction.

\subsection{Looking for Structure} Once we are given such a graph, how are we supposed to see whether there is any `structure' present? Spectral Graph Theory provides a simple way of measuring the `structuredness' of such a graph (and thus of such a sequence) in a specific way. Given such a graph $G=(V,E)$, the adjacency matrix $A$ is an $n \times n$ matrix given by
$$ A_{ij} = \# \mbox{edges between}~x_i~\mbox{and}~x_j.$$
The matrix $A$ is symmetric and has real eigenvalues. Since the graph is 4-regular, the largest eigenvalue is 4. We will order the eigenvalues by decreasing absolute value, i.e.
$$ 4 = \lambda_1 \geq |\lambda_2| \geq |\lambda_3| \geq \dots \geq |\lambda_n| \geq 0.$$
\begin{quote}
\textbf{General Rule.} We have, for any 4-regular graph,
$$ 3.46 \dots \sim \sqrt{12} \pm \varepsilon \leq |\lambda_2| \leq 4.$$
Moreover, if the sequence $x_1, \dots, x_n$ is random, then the associated graph has $|\lambda_2|$ close to $\sqrt{12}$. Conversely, being close to 4 is indicative of some structure being present.
\end{quote}
To be more precise, if $x_1, \dots, x_n$ are i.i.d. random variables, it is a result of Friedman \cite{fried} that the graph will be close to a Ramanujan expander (meaning $|\lambda_2| \leq \sqrt{12} + \varepsilon$) with high probability as $n$ tends to infinity for any $\varepsilon > 0$. The Alon-Boppana bound \cite{nilli} shows that $|\lambda_2|$ cannot be substantially smaller than $\sqrt{12}$. As for the upper bound, we remark that $\lambda_2$ being very close to 4 would imply the existence of structure as follows: using the Rayleigh quotient,
$$ 4 - \lambda_2 = \inf_{f \in \mathbb{R}^n \atop \left\langle f, 1 \right\rangle = 0} \frac{ \sum_{(u,v) \in E} (f(u)-f(v))^2}{ \sum_{v \in V}{ f(v)^2}},$$
where $f:V \rightarrow \mathbb{R}$ ranges over all vectors of size $n$ that have mean value 0. This quantity being small means that it is possible to separate $V$ into two parts, the sets $\left\{v \in V: f(v) \leq 0\right\}$ and $\left\{v \in V: f(v) > 0\right\}$ such that there are relatively few edges that run between the two sets. If we know that $x_i$ is in one of the sets, then $x_{i-1}$ and $x_{i+1}$ are also likely to be in the set -- and so are elements that are just slightly larger or smaller. If $\lambda_2$ is close to $-4$, then $|\lambda_2| \sim 4$ and a similar reasoning shows the existence of an almost bipartite structure. In either case, there is a lot of structure in the graph and thus the sequence. We refer to \cite{chung, linial, kow, lub} for nice introductions. The main idea in \cite{steinerberger} was to check whether $|\lambda_2| \sim \sqrt{12}$ and to use this as a way of measuring whether a sequence of $n$ distinct reals behaves `like a sequence of i.i.d. random variables'.  The idea in this paper is the other direction, see whether  $|\lambda_2| \sim 4$ and to use this as a way of detecting structure. In practice, most examples in this paper
were so visually striking that we did not require any type of spectral test -- a notable exception being the EKG sequence in \S \ref{sec:EKG}.  

\section{A List of Graphs, Sequences and Open Problems}
\textbf{Introduction.} This section contains the core of the paper: a series of sequences for which the associated graphs were found to have have intricate, nontrivial behavior. The first two examples (\S 2.1 and \S 2.2) were already discussed in \cite{steinerberger}, all other examples are new. We were surprised at the wealth of examples and believe that there are many more. The complete list of sequences discussed in this paper, indexed by their number in the Online Encyclopedia of Integer Sequences (OEIS) when available, is as follows.

\begin{center}
\begin{table}[h!]
\begin{tabular}{c | c |  c}
Section & OEIS number & first defined  \\
\hline
2.1 &  not integer sequence & Kronecker/Weyl \\
2.2 &  not integer sequence & van der Corput (1935)\\
2.3 &  A036552 & Erd\H{o}s, Freud \& Hegyvari (1983) \\
2.3 &  A064736 & Erd\H{o}s, Freud \& Hegyvari (1983) \\
2.4 & A076641 &  Jooste (2002) \\
2.5 & A053985  & Schroeppel (1972)\\
2.5 & A065369  &  LeBrun (2001)\\
2.5 & A073791--A073796  & Wilson (2002)\\
2.5 & A073835 & Wilson (2002)\\
2.6 & A339571 & this paper (from A133058)\\
2.7 & A064413 & Ayres (2002)\\
2.8 & not integer sequence & this paper (from A014486) \\
2.9 & A005132 & Recaman Santos (1991) \\
2.10 & A127202 & Quet (2007) \\
2.11 & A281488 & Zabolotskiy (2017)\\
2.12 & A006068 & described by Gardner in 1972\\
2.13 & A140589 & Curtz (2008)\\
2.14 & A059893 & related to Stern-Boroct, Calin-Wilf\\
2.15 &  A347520	 & this paper (based on A053392)\\
2.16 & A057163 & Deutsch (1998), Karttunen (2000)
\vspace{10pt}
\end{tabular}
\caption{Sequences discussed in this paper per OEIS number.}
\end{table}
\end{center}
\vspace{-20pt}

We tried to trace the origin of the sequences using information provided in the OEIS, when available. To the best of our knowledge, we define two new sequences
(now A339571 and A347520) which both arise from taking existing sequences (A133058 and A014486, respectively) and removing duplicates.\\

\textbf{A Quick Word on Methodology.} Our examples are mostly experimental: we took the first $n$ elements of a sequence, construct and then investigate the graph. Usually, we investigated the $500 \leq n \leq 5000$ range for each sequence: once a sequence seems to lead to interesting graphs, the value of $n$ does not seem to matter very much. Our interpretation is that it's actually hard to produce interesting graphs `by accident' (even for small $n$) and that sequences have to have a good reason for doing so (and continue having this reason as $n$ grows).  Naturally, however, it is conceivable that graphs associated to a sequence become unstructured after a certain point. Most of our examples (2.1, 2.2, 2.4, 2.5, 2.8, 2.9, 2.12, 2.13, 2.14, 2.15, 2.16) exhibit such structured graphs that one would expect strong structural statements to be the cause of these graphs. Several other examples are so mysterious (2.3, 2.6, 2.7, 2.10, 2.11) that it is difficult to say what one would expect.\\
We visualized the graphs using Mathematica's \textsc{Graph} command. Embedding graphs into Euclidean space is a nontrivial task and it is conceivable that by using other graph embedding techniques more or additional structure can be found. Most of our examples are so striking that one would expect the embedding algorithm to not play much of a role (we refer to \S 3.3 for some more comments on this). 

\begin{center}
\begin{figure}[h!]
\begin{tikzpicture}
\node at (0,0) {\includegraphics[width=0.25\textwidth]{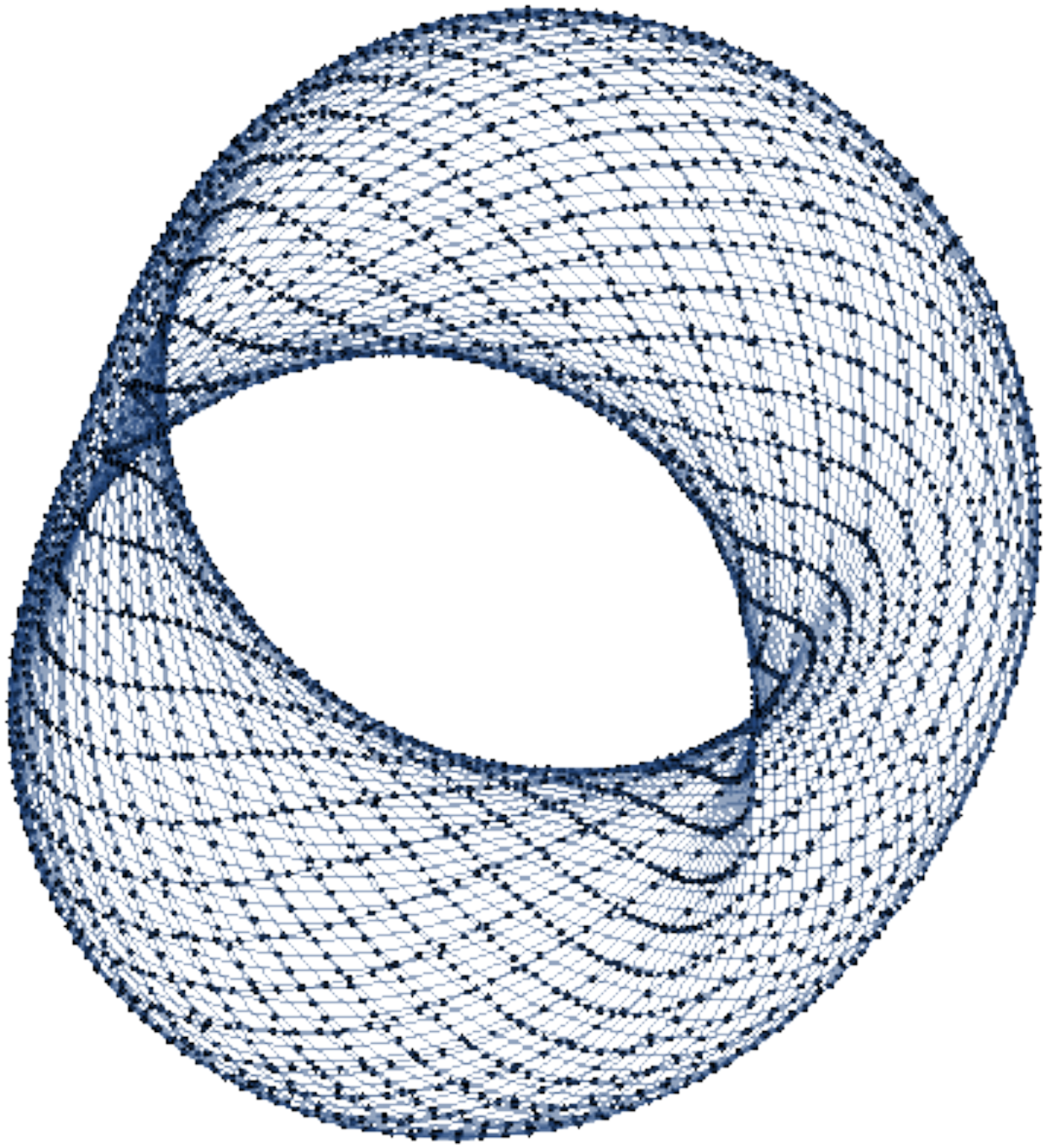}};
\node at (5,0) {\includegraphics[width=0.25\textwidth]{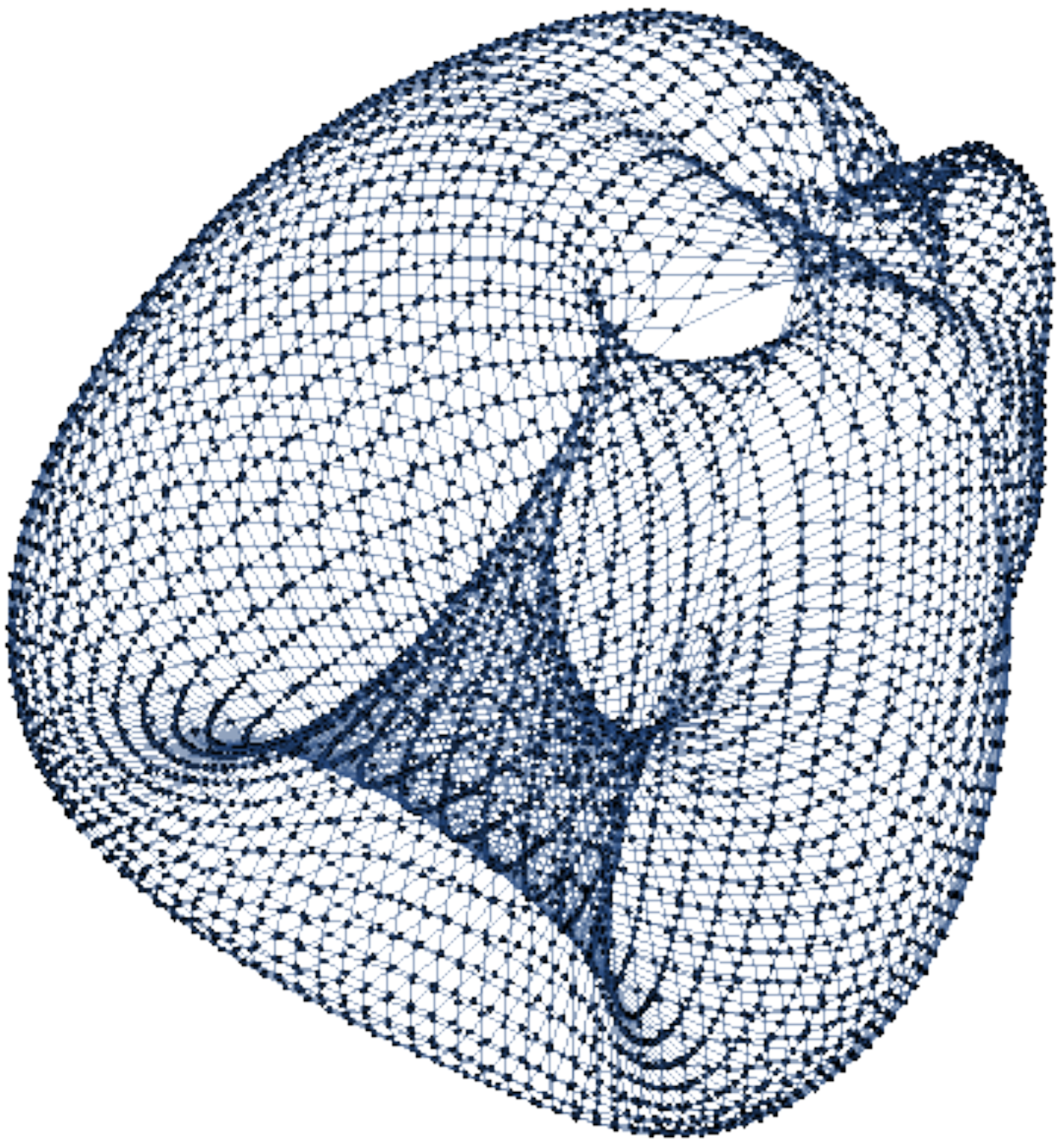}};
\end{tikzpicture}
\caption{The Kronecker sequence with $\alpha$ being the golden ratio (left) and the van der Corput sequence in base 2 (right).}
\label{fig:known}
\end{figure}
\end{center}

\vspace{-10pt}

\subsection{The Kronecker sequence}
Let $\alpha \in \mathbb{R}$ be irrational and let $a_n = \left\{n \alpha \right\}$, where $\left\{ \cdot \right\}$ denotes its fractional part, i.e. $\left\{ x \right\} = x - \left\lfloor x \right\rfloor$.
It was already discussed in \cite{steinerberger} that this sequence leads to an interesting torus-type structure. The `irrationality' of $\alpha$ seems to play a role (i.e. it seems to matter how well $\alpha$ can be approximated by rationals). We show an example for $\alpha = (1+\sqrt{5})/2$ in Fig. \ref{fig:known} where one can observe a two-dimensional object that seems to be nicely embedded in three dimensions: a torus with an additional twist.

\subsection{The van der Corput sequence} The van der Corput sequence (in base 2) is another classical object in discrete mathematics.  Its definition is somewhat whimsical: the element $a_n$ is obtained by writing $n$ in base 2, then reversing the order of the digits and then interpreting the resulting element as a real in $[0,1]$. This leads to the sequence (numerator and denominator given by the OEIS sequences A030101 and A062383)
$$ \frac{1}{2}, \frac{1}{4}, \frac{3}{4}, \frac{1}{8}, \frac{5}{8}, \frac{3}{8}, \frac{7}{8}, \frac{1}{16}, \dots$$
It is known for having optimal behavior with respect to how well distributed it is: for any $n \in \mathbb{N}$, the elements $a_1, \dots, a_n$ are fairly evenly distributed throughout $[0,1]$ in a way that can be made precise (we refer to the classical textbooks \cite{dick, drmota, kuipers} and the result of Schmidt \cite{schmidt}). The arising graphs were already discussed in \cite{steinerberger} and seem to have interesting and highly nontrivial structure, we refer to Fig. \ref{fig:known}.

\subsection{Two Sequences of Erd\H{o}s, Freud \& Hegyvari} In a 1983 paper, Erd\H{o}s, Freud \& Hegyvari \cite{erd} define the following nice sequence (A064736 in the Online Encyclopedia of Integer Sequences (OEIS)): we set $a_1 = 1, a_2=2$ and then define $a_{2n+2}$ as the smallest integer not appearing in $\left\{a_1, \dots, a_{2n}\right\}$ and $a_{2n+1} = a_{2n} \cdot a_{2n+2}$. The sequence starts like
$$ 1, 2, 6, 3, 12, 4, 20, 5, 35, 7, 56, 8, 72, \dots$$
This sequence has the interesting property that
$$ \lim \inf_{n \rightarrow \infty} \frac{\gcd(a_n, a_{n+1})}{n} \geq \frac{1}{2}.$$
The same paper defines a second sequence that is similar (A036552): $a_1 =1, a_2=2$ and, generally, $a_{2n}$ is the smallest integer that has not yet appeared while $a_{2n+1} = 2a_{2n}$. The sequences do not seem to have been studied very much: we found them mentioned in a survey article written by Neil Sloane for the \textit{Notices} \cite{notices}. Both sequences result in very interesting graphs (see Fig. \ref{fig:erd}).

\begin{center}
\begin{figure}[h!]
\begin{tikzpicture}
\node at (0,0) {\includegraphics[width=0.45\textwidth]{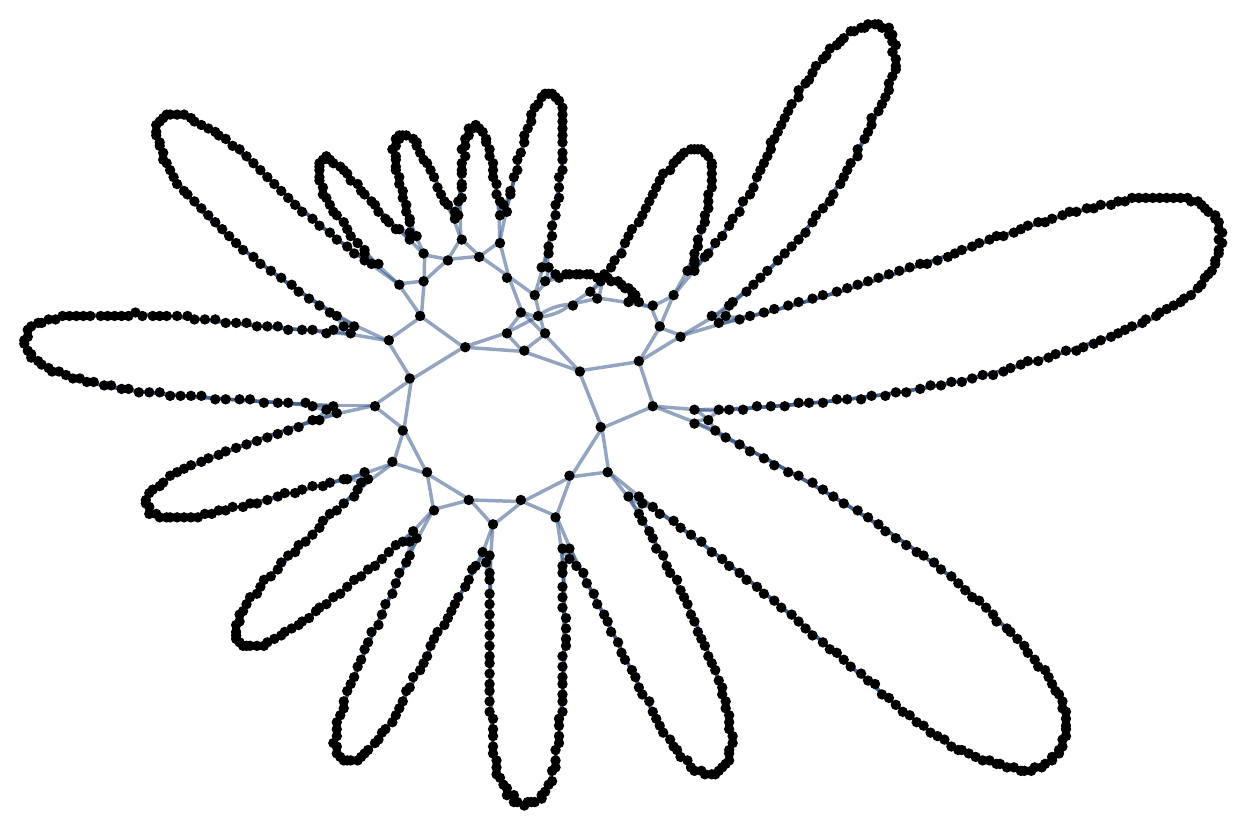}};
\node at (6,0) {\includegraphics[width=0.4\textwidth]{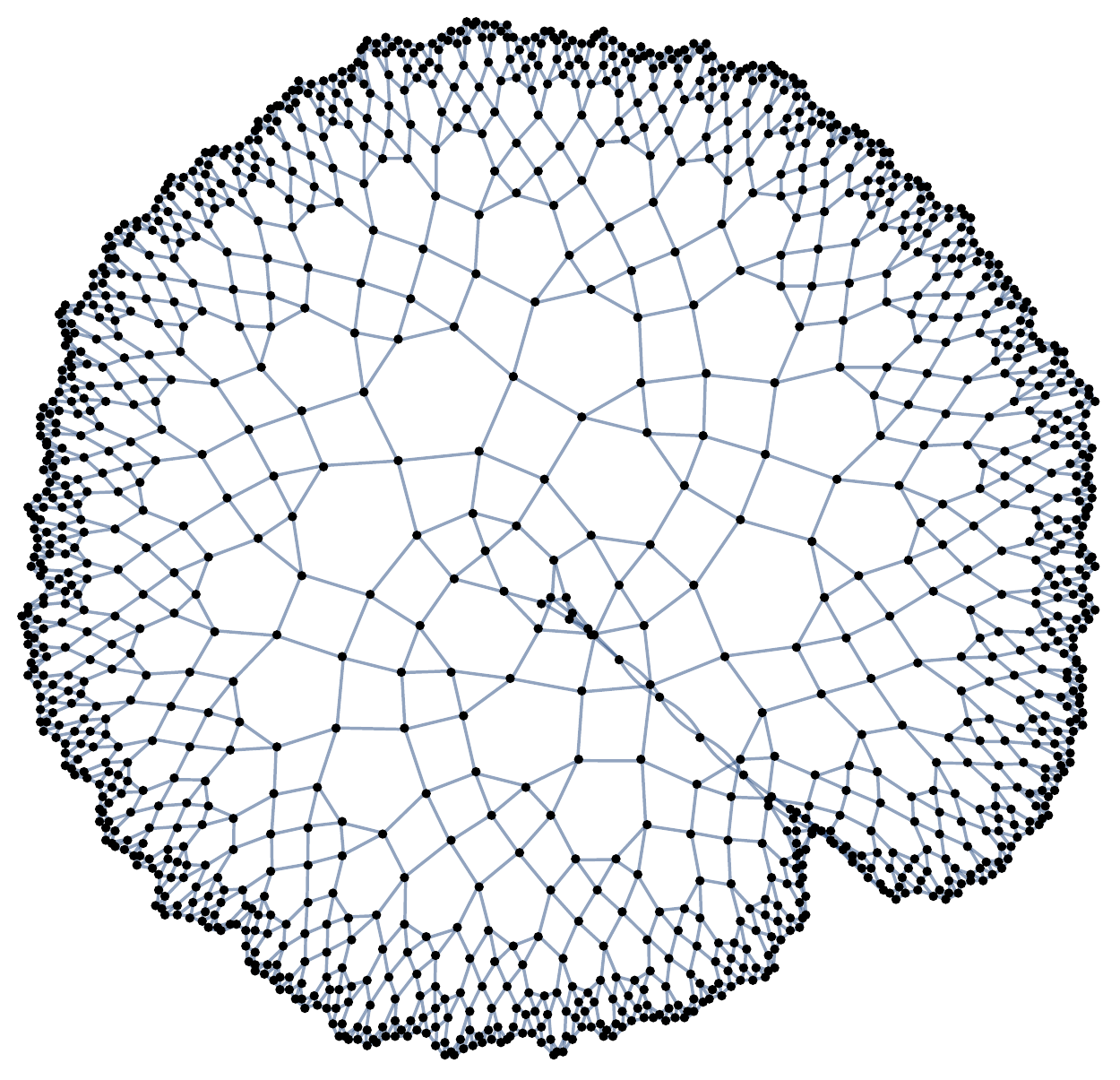}};
\end{tikzpicture}
\caption{Graphs for A064736 (left) and A036552 (right).}
\label{fig:erd}
\end{figure}
\end{center}

\subsection{The Reversal Sequence with duplicates removed} This sequence is derived from the reversal sequence A004086. The reversal sequence has a very simple definition: for the \textit{n-th} element $a_n$, simply read the digits of the integer backwards (in base 10). The sequence starts $0,1, \dots, 9, 1, 21, 31, 41, \dots$. We remove all duplicates of the sequence and thus arrive at A076641 in the OEIS:
$$ 0, 1, 2, 3, 4, 5, 6, 7, 8, 9, 11, 21, 31, 41, 51, 61, 71, 81, \dots$$
The arising graph structure is quite curious (see Fig. \ref{fig:rev}). Given the fairly explicit definition of the sequence, one might hope that this phenomenon can be analyzed.

\begin{center}
\begin{figure}[h!]
\begin{tikzpicture}
\node at (0,0) {\includegraphics[width=0.4\textwidth]{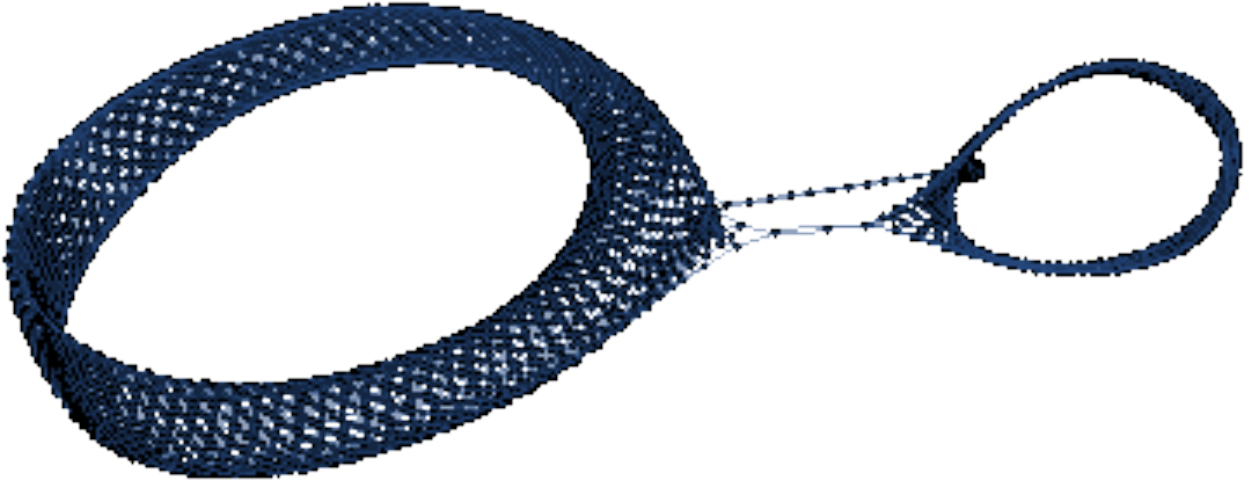}};
\node at (6,0) {\includegraphics[width=0.4\textwidth]{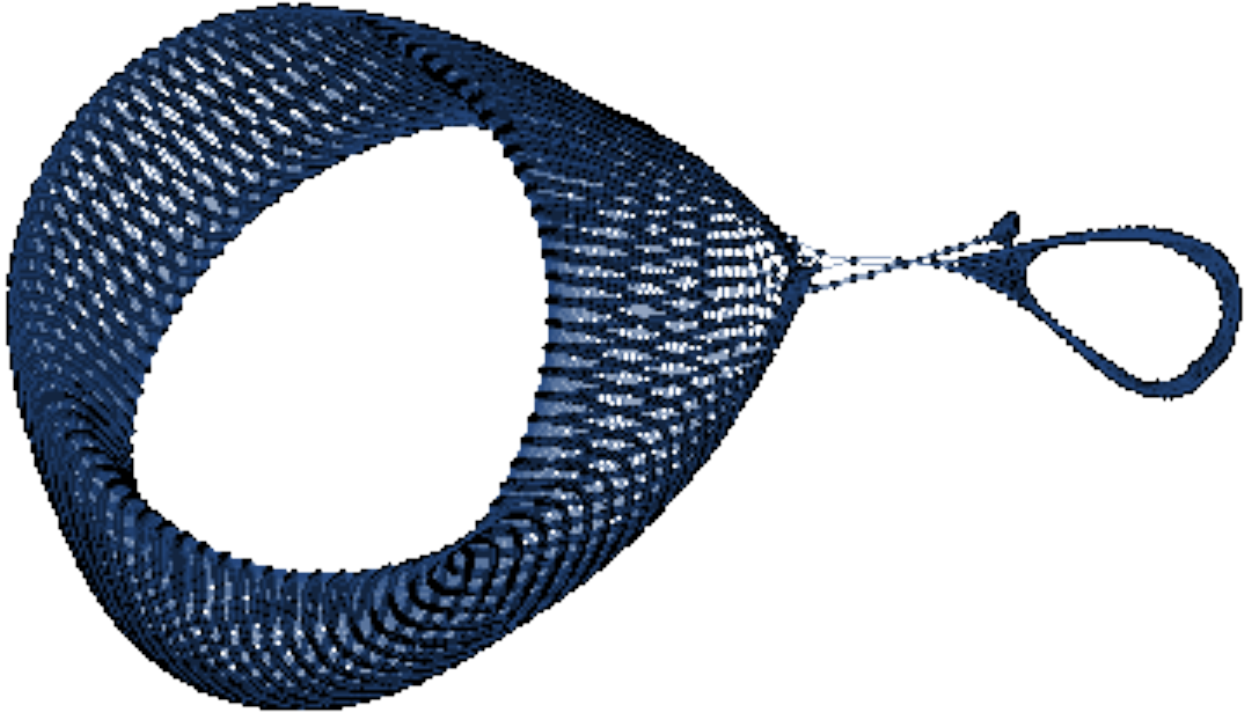}};
\end{tikzpicture}
\caption{Graphs arising from the reversal sequence for the first 5000 elements (left) and the first 10000 elements (right).}
\label{fig:rev}
\end{figure}
\end{center}

\subsection{A Sign Flipping Sequence} In this section we discuss a type of sequence indexed by base $b \in \mathbb{N}$, $b\geq 2$. The definition is quite simple: in base $b$, the \textit{n-th} element is determined by the expansion of the integer $n$ in base $b$: if
$$ n = \sum_{k=0}^{\infty} d_k b^k, \quad \mbox{then} \quad a_n =  \sum_{k=0}^{\infty} d_k (-b)^k.$$
Illustrating it with an explicit example, in base 2 we have 
$$ 9 = 1001_2 \qquad \mbox{and thus} \qquad a_9 = -2^3 + 2^0 = -7.$$
In particular, the sequence (in base 2, A053985) begins 
$$ 0, 1, -2, -1, 4, 5, 2, 3, -8, -7, -10, -9, -4, -3 \dots$$
These sequences exist in the OEIS, at least for base $b=3$ (A065369), $4 \leq b \leq 9$ (A073791--A073796) and base $b=10$ (A073835). However, it does not seem as if this sequences have been studied very much -- we show that they exhibit fairly intricate structures.
We note that these sequences are injective, it is not possible for the same element to appear more than once (something that is required for our graph construction: the numbers have to be distinct).
\begin{lemma} For any $b \geq 2$, the associated sequence is injective.
\end{lemma}
\begin{proof}
Suppose that $a_{n_1} = a_{n_2}$. We expand the integers into base $b$ and obtain
$$ n_1 = \sum_{k=0}^{\infty} d_k b^k \qquad \mbox{and} \qquad  n_2 = \sum_{k=0}^{\infty} e_k b^k.$$
By assumption, we have
$$ \sum_{k=0}^{\infty} d_k (-b)^k  = \sum_{k=0}^{\infty} e_k (-b)^k.$$
By separating into even and odd powers, we have
$$ \sum_{k=0}^{\infty} (d_{2k} - e_{2k}) b^{2k} =  \sum_{k=0}^{\infty} (d_{2k+1} - e_{2k+1}) b^{2k+1}.$$

\begin{center}
\begin{figure}[h!]
\begin{tikzpicture}
\node at (0,0) {\includegraphics[width=0.4\textwidth]{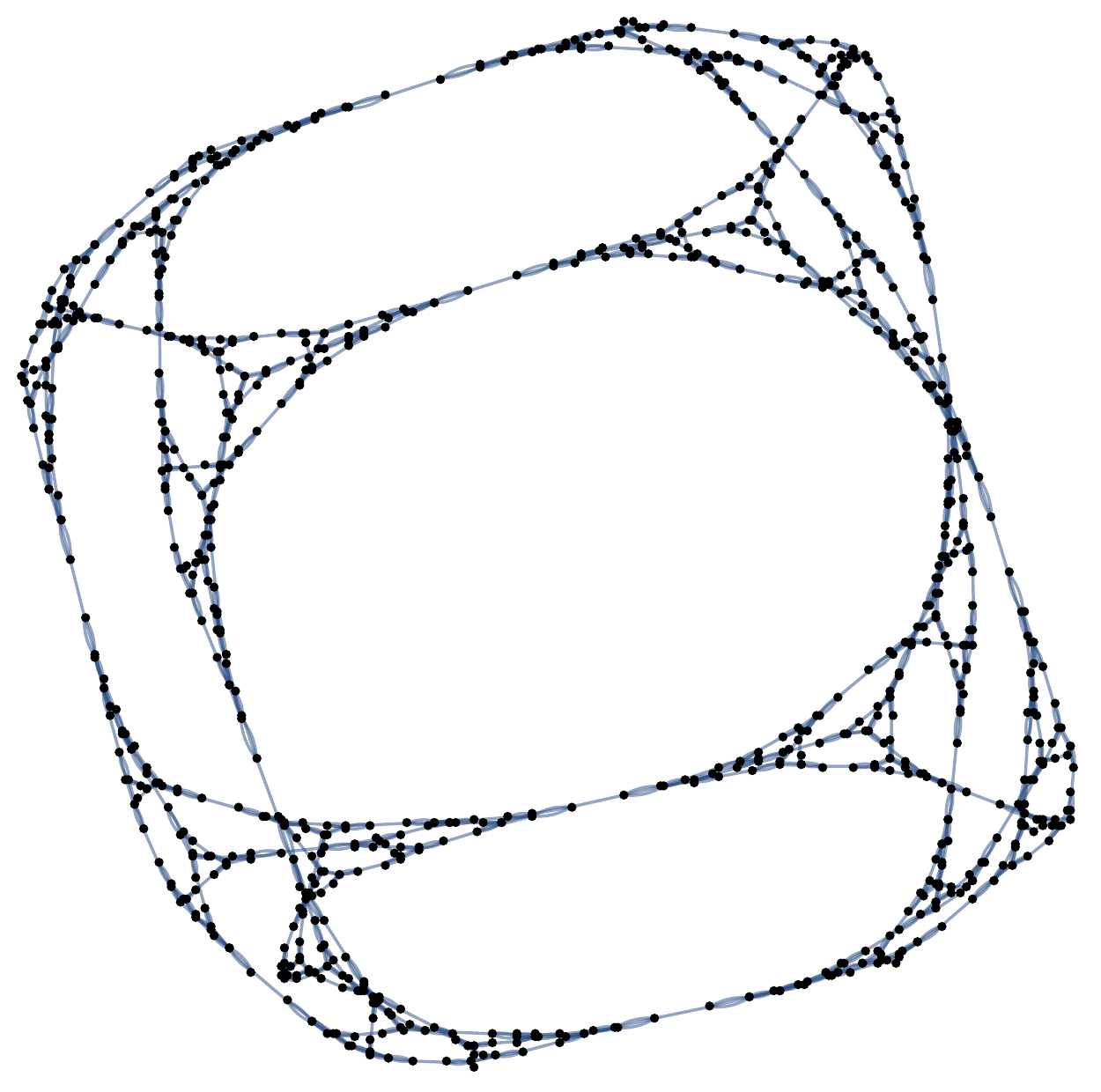}};
\node at (6,0) {\includegraphics[width=0.4\textwidth]{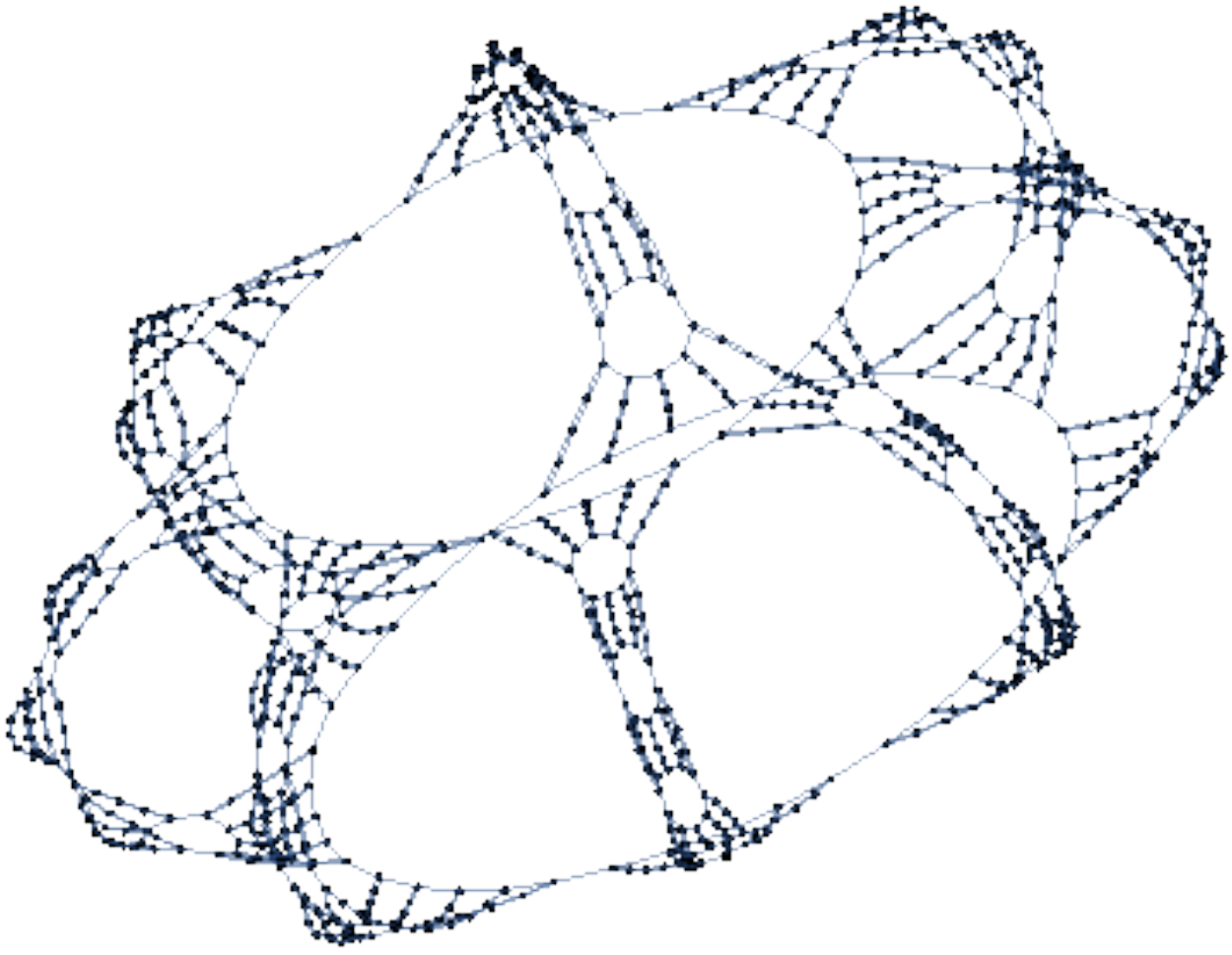}};
\end{tikzpicture}
\caption{Graphs for the first few elements in base 2 (left) and 5 (right).}
\end{figure}
\end{center} 

Since $n_1 \neq n_2$, the $d_i$ and $e_i$ have to differ in at least one place. Let us first assume that the largest index at which they differ is $2\ell$. Then
 $$ \left| \sum_{k=0}^{\infty} (d_{2k} - e_{2k}) b^{2k} \right| \geq b^{2\ell} - \sum_{k=0}^{\ell-1} (b-1) b^{2k} = b^{2\ell} - \frac{b^{2\ell} - b^2}{b+1} \geq \frac{b}{b+1} b^{2\ell}.$$
 At the same time, the right-hand side is at most
 $$  \left| \sum_{k=0}^{\infty} (d_{2k+1} - e_{2k+1}) b^{2k+1} \right| \leq  \left| \sum_{k=0}^{\ell-1} (b-1) b^{2k+1} \right| = \frac{b}{b+1} b^{2\ell} - \frac{b}{b+1}.$$

 This shows that these two quantities can never coincide. The first significant difference in a digit has thus to occur at index $2\ell+1$. By the same type of argument, we have
 $$  \left| \sum_{k=0}^{\ell} (d_{2k+1} - e_{2k+1}) b^{2k+1} \right| \geq b^{2\ell+1} -  \sum_{k=0}^{\ell-1} (b-1) b^{2k+1} \geq \frac{b}{b+1} b^{2\ell+1}$$
 while
  $$ \left| \sum_{k=0}^{\ell} (d_{2k} - e_{2k}) b^{2k} \right| \leq   \left| \sum_{k=0}^{\ell} (b-1) b^{2k} \right| = \frac{b}{b+1}b^{2\ell+1} - \frac{1}{b+1}.$$
  This shows the injectivity of the sequence in every base.
\end{proof}
\begin{proof}[Yet another proof of Lemma 1]
Let $b\ge 2$ be a fixed base. Note that there is a unique polynomial $p_n$ with coefficients in $\{0, \ldots, b-1\}$ such that $n=p_n(b).$ With this notation, it is easy to see that $a_n=p_n(-b).$ Suppose $a_{n}=a_{m},$ and let $p_n(x)=\sum_{i=0}^{k_n}d_ix^i$ and $p_m=\sum_{i=0}^{k_m}e_ix^i.$ Note that $a_n=a_m\iff p_n(-b)=p_m(-b),$ and evaluating both sides modulo $b,$ we obtain $d_0=e_0.$ Given $d_i=e_i, 0\le i\le j-1,$ it follows that \[(-b)^{-j}\left(p_n(-b)-\sum_{i=0}^{j-1}d_i(-b)^i\right)=(-b)^{-j}\left(p_m(-b)-\sum_{i=0}^{j-1}e_i(-b)^i\right).\]
Evaluating both sides of the above expression, we obtain $d_j=e_j.$ It follows, inductively, that $p_n=p_m.$
\end{proof}

  \begin{center}
\begin{figure}[h!]
\begin{tikzpicture}
\node at (0,0) {\includegraphics[width=0.4\textwidth]{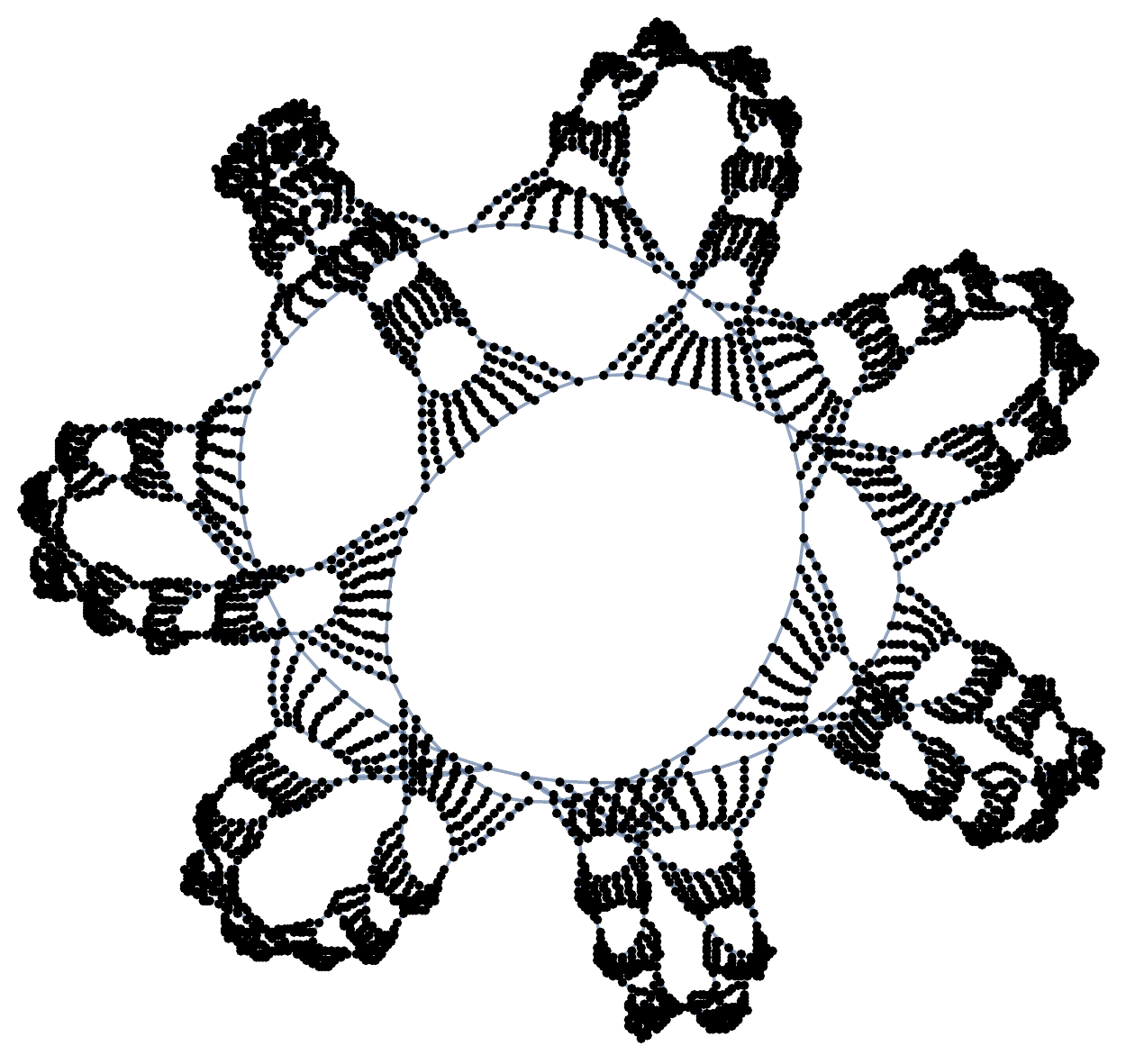}};
\node at (6,0) {\includegraphics[width=0.45\textwidth]{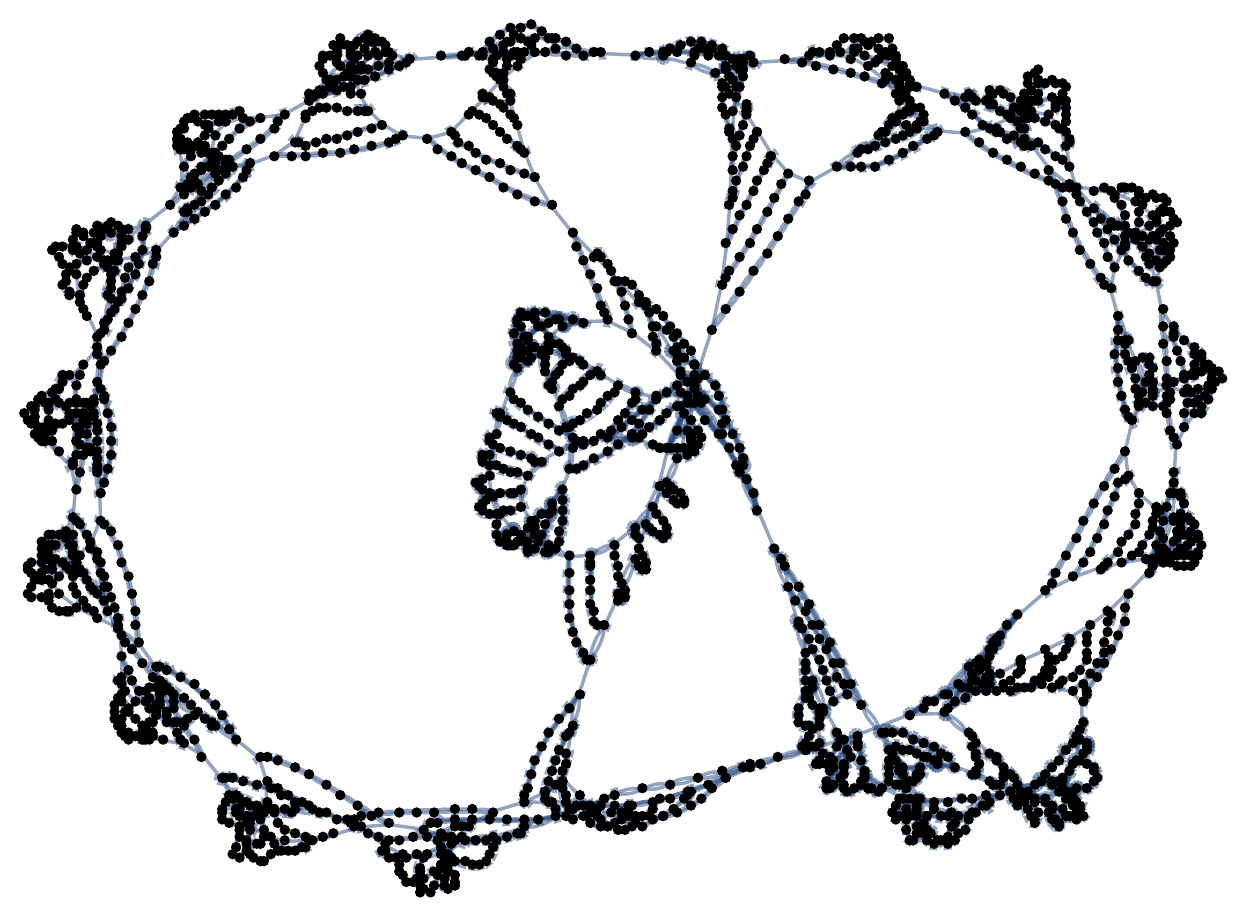}};
\end{tikzpicture}
\caption{Graphs for the first few elements in base 9 (left) and 13 (right).}
\end{figure}
\end{center}

\begin{lemma} Let $b\ge 2$ be a fixed base and let $a_n$ be the sequence as defined in \S 2.5 with respect to a basis $b$. Then, $a_n$ satisfies the recurrence relation \[a_{bk+m}=(-b)a_k+m.\]
\end{lemma}
\begin{proof}
Let $n\in \mathbb{N}.$ We can write $n=\sum_{i=0}^{N}x_ib^{i}=bk+x_0$ where $k=\sum_{i=1}^{N-1}x_ib^{i-1}.$ From the definition of the sequence we have 
\begin{align*}
a_n &= (-b)\sum_{i=1}^{N-1}x_i(-b)^{i-1}+x_0\\
&=(-b)a_k+x_0.
\end{align*}
This is the desired statement.
\end{proof}

One would perhaps assume that these intricate graphs are the consequence of iteratively applying Lemma 2 in some fashion. It is not clear to us at this point and further investigating these graphs seems like an interesting avenue.

\subsection{A GCD sequence} This sequence is derived from A133058: we consider A133058 and remove all duplicates resulting in the sequence A339571. A133058 (proposed by C. Zizka) has an iterative definition: $a_0 = a_1 = 1$ and, for $n \geq 2$, 
$$ a_{n} = \begin{cases} a_{n-1} + n + 1 \qquad &\mbox{if}~\gcd(a_{n-1},n) = 1\\
\frac{a_{n-1}}{\gcd(a_{n-1}, n)} \qquad &\mbox{otherwise.} \end{cases}$$
We learned about the sequence from a numberphile video \cite{number2}. 

  \begin{center}
\begin{figure}[h!]
\begin{tikzpicture}
\node at (0,0) {\includegraphics[width=0.4\textwidth]{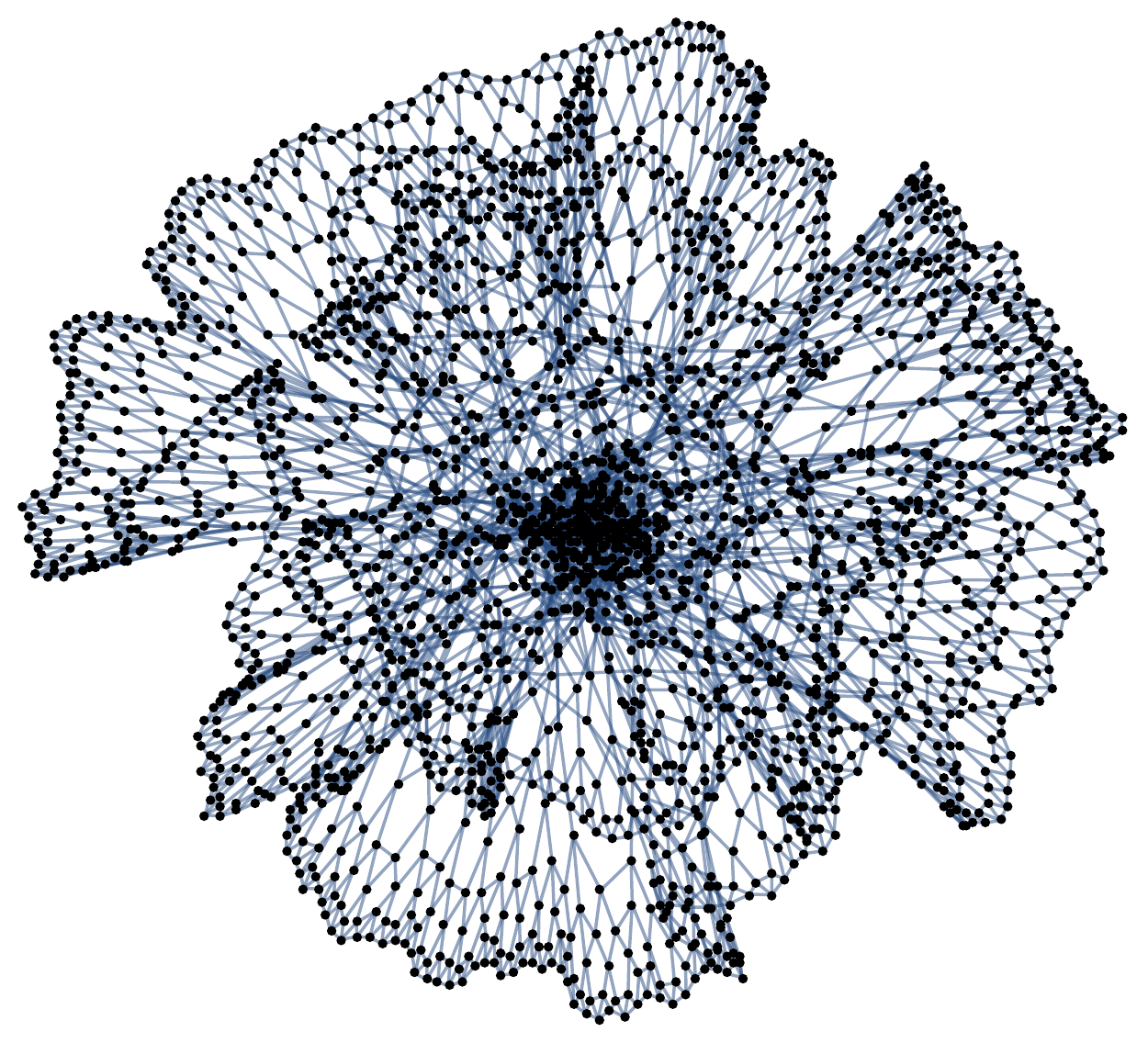}};
\node at (6,0) {\includegraphics[width=0.4\textwidth]{butterfly.pdf}};
\end{tikzpicture}
\caption{A133058 with duplicates removed.}
\end{figure}
\end{center}

\subsection{The EKG sequence} \label{sec:EKG} The EKG sequence, first introduced by Jonathan Ayres (A064413, \cite{ekg}), is a famously tricky sequence. Its definition is simple: $a_1 = 1$, $a_2 = 2$ and for each $n \geq 3$, $a_n$ is the smallest natural number that has not yet appeared with the property that
$$ \gcd(a_n, a_{n-1}) \geq 2.$$
The sequence begins $1,2,4,6,3,9,12,8, \dots$.  The sequence is known to be a permutation of the integers and has linear growth $c n \leq a_n \leq Cn$ (see \cite{ekg}). At first glance, the graph seems rather random (see Fig. \ref{fig:ekg}) but it does reveal itself in having an eigenvalue suspiciously close to 4 (for the first 5000 terms, the eigenvalue is 3.96). Embedding the graph into $\mathbb{R}^3$, we observe a type of parabola or cone structure. It is not clear to us what will happen as the graphs get larger and larger.
\begin{center}
\begin{figure}[h!]
\begin{tikzpicture}
\node at (0,0) {\includegraphics[width=0.4\textwidth]{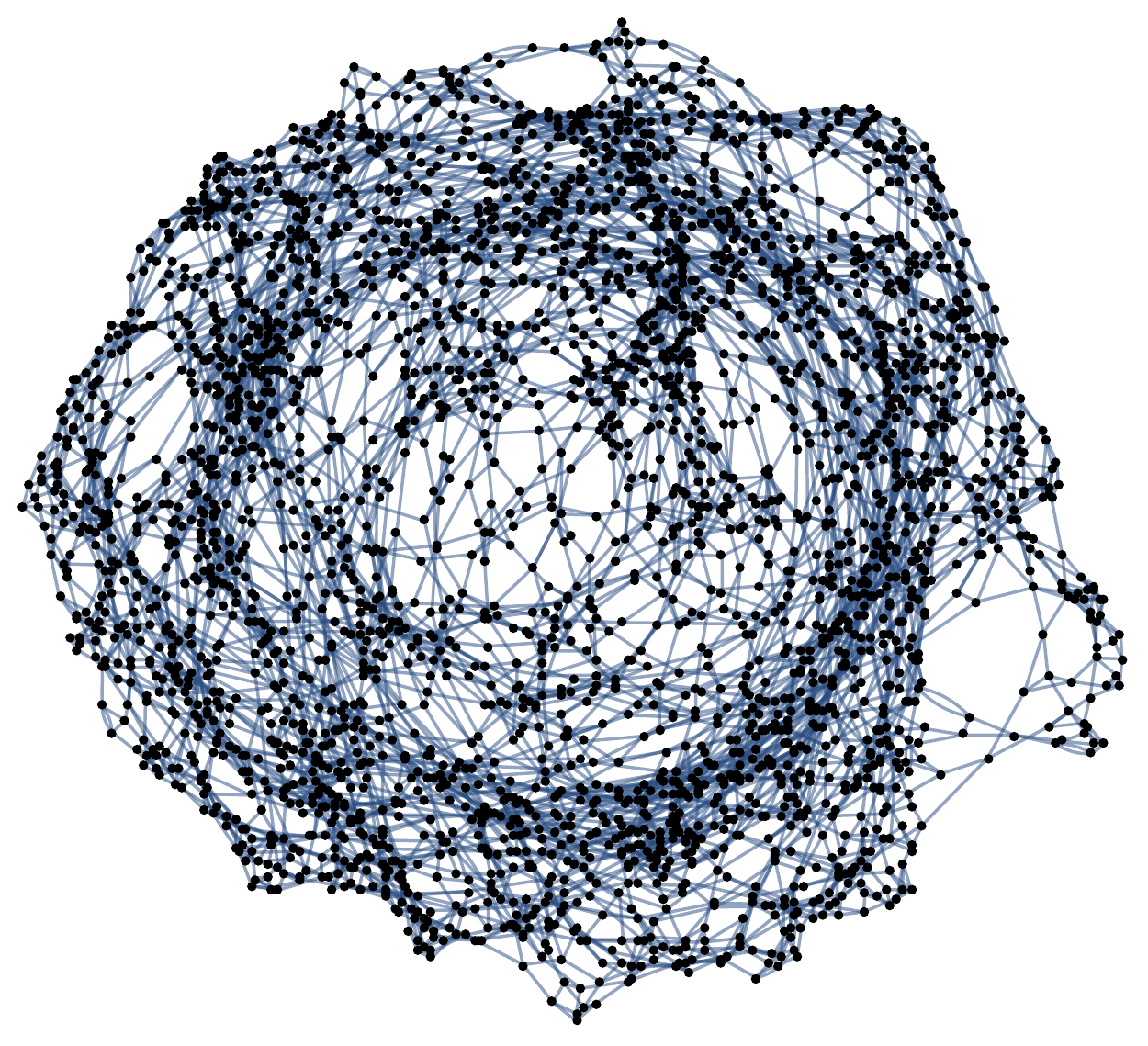}};
\node at (5,0) {\includegraphics[width=0.4\textwidth]{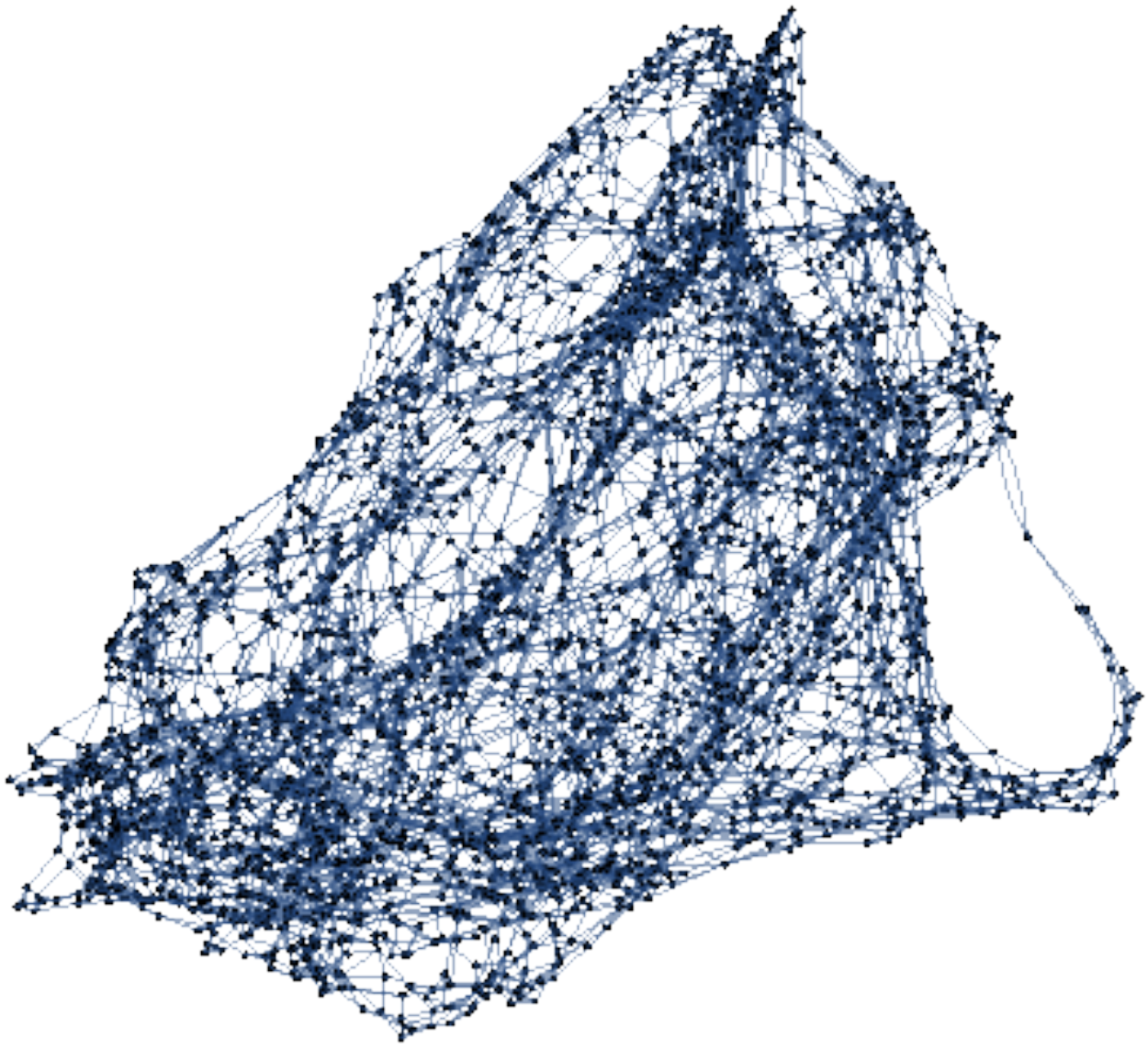}};
\end{tikzpicture}
\caption{The Graph associated to the first few elements of the EKG sequence, an embedding in two dimensions (left) and an embedding in three dimensions (right).}
\label{fig:ekg}
\end{figure}
\end{center}

\subsection{Totally Balanced Integers} This example is a subsequence of the van der Corput sequence (in base 2). A014486 is the sequence of totally balanced sequence of integers whose binary expansion has an even length, there are equally many 0's and 1's and, starting from the most significant bit, the number of 0's never exceeds the number of 1's. The first few elements, written in binary, are 0, 10, 1010, 1100, 101010, ... which, in base 10, becomes the sequence
$$ 0, 2, 10, 12, 42, 44, 50, 52, 56, 170, \dots$$
If we denote the van der Corput sequence (in base 2) by $v_n$ and the sequence of totally balanced sequences of integers by $A014486_n$, then our example is given by
$ a_n = v_{A014486_n}.$
The sequence begins
$$ \frac{1}{4}, \frac{5}{16}, \frac{3}{16}, \frac{21}{64}, \frac{13}{64}, \frac{19}{64}, \frac{11}{64}, \frac{7}{64}, \dots$$
Its numerators are A072800 in the OEIS, the sequence of denominators are A339570.
We find that the associated graph has a rather striking appearance indicating some highly nontrivial dynamics in the integer sequence. It seems like an interesting avenue to understand when suitable subsequences of a structured sequences inherit some of the regularity in the structure? 
It was observed by Neil Sloane (personal communication) and verified by Hugo Pf\"ortner for the first million terms that the denominators seem to powers of $4$ and that the number of times $4^k$ appears is given by the $k-$th Catalan number. This is indeed the case.

\begin{proposition} The sequence of denominators of the sequence defined in \S 2.8 is given by powers of $4$. $4^k$ appears 
$$ C_k = \frac{1}{2k+1} \binom{2k}{k}$$
many times.
\end{proposition}
\begin{proof} Reversing the digit expansion of a totally balanced integer always results in an odd numbers. Moreover, the set of totally balanced integers with $2n$ digits has $C_n$ elements: since the fractions are primitive, the result follows. 
\end{proof}

\begin{center}
\begin{figure}[h!]
\begin{tikzpicture}
\node at (5,0) {\includegraphics[width=0.4\textwidth]{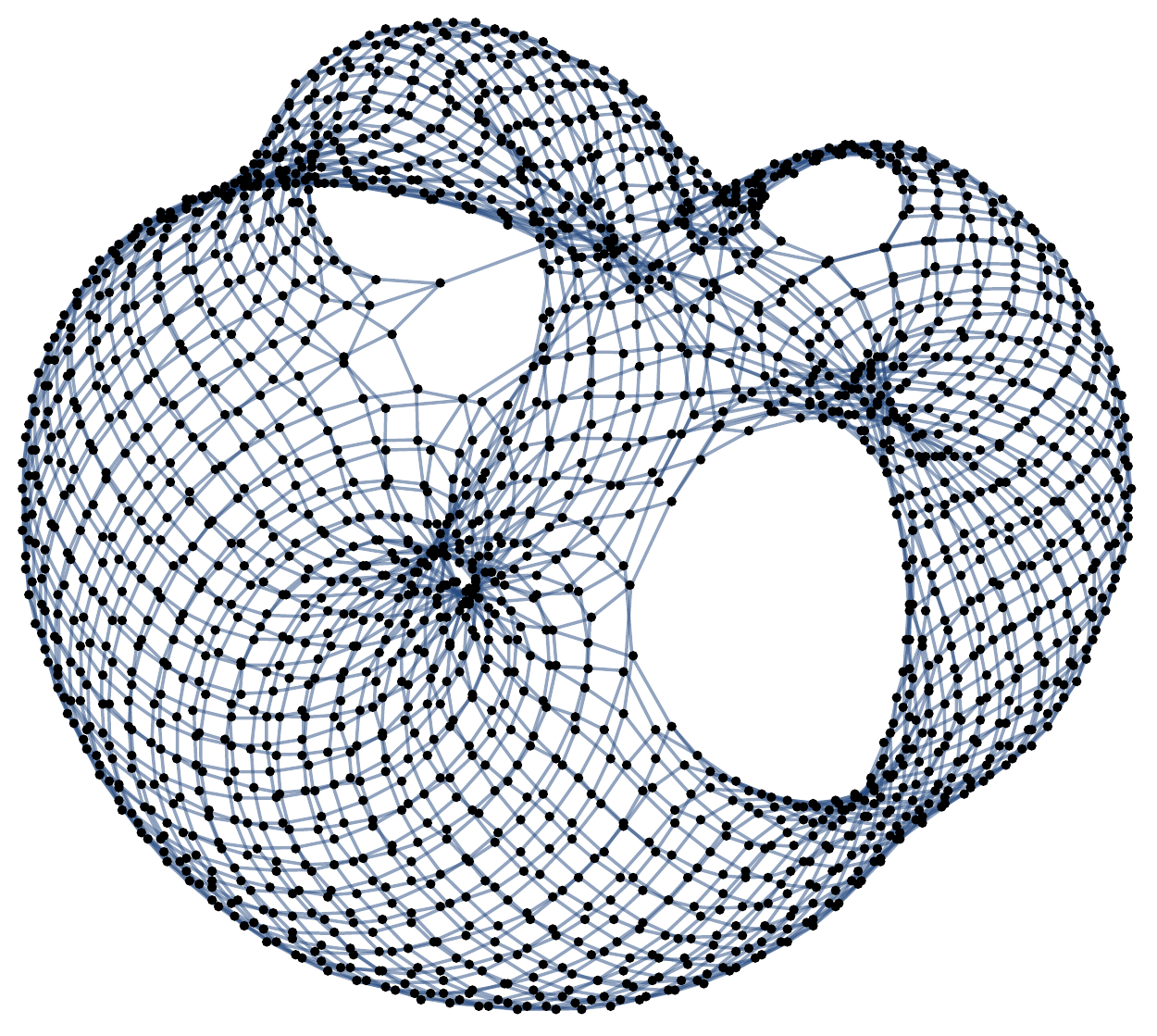}};
\end{tikzpicture}
\caption{Graph associated to the van der Corput sequence along the subsequence of totally balanced integers (A014486).}
\label{fig:sub}
\end{figure}
\end{center}

\subsection{The Recaman sequence} The Recaman sequence (A005132), first defined by Bernardo Recaman Santos in a letter to N. J. A. Sloane in 1991, is a well-known sequence (being featured in both a numberphile video \cite{bellos} as well as in Sloane's `Favorite Sequences' \cite{sloane}). It is given by $a_0 = 0$ and
$$ a_n = \begin{cases} a_{n-1} -n \qquad &\mbox{if}~a_{n-1} > n~\mbox{and not already in the sequence}\\ 
a_{n-1} + n \qquad &\mbox{otherwise.} \end{cases}$$
The sequence begins $0, 1, 3, 6, 2, 7, 13, 20, 12, 21, \dots$. It is not a permutation of the integers since $a_{20} = 42 = a_{24}$. Sloane conjectured that each number will eventually appear in the sequence but this is not known. Graphing the sequences after removing all duplicates results in a quite interesting graph. It is difficult to make a very precise statement but the graph immediately shows the presence of extraordinary structure.

\begin{center}
\begin{figure}[h!]
\begin{tikzpicture}
\node at (0,0) {\includegraphics[width=0.5\textwidth]{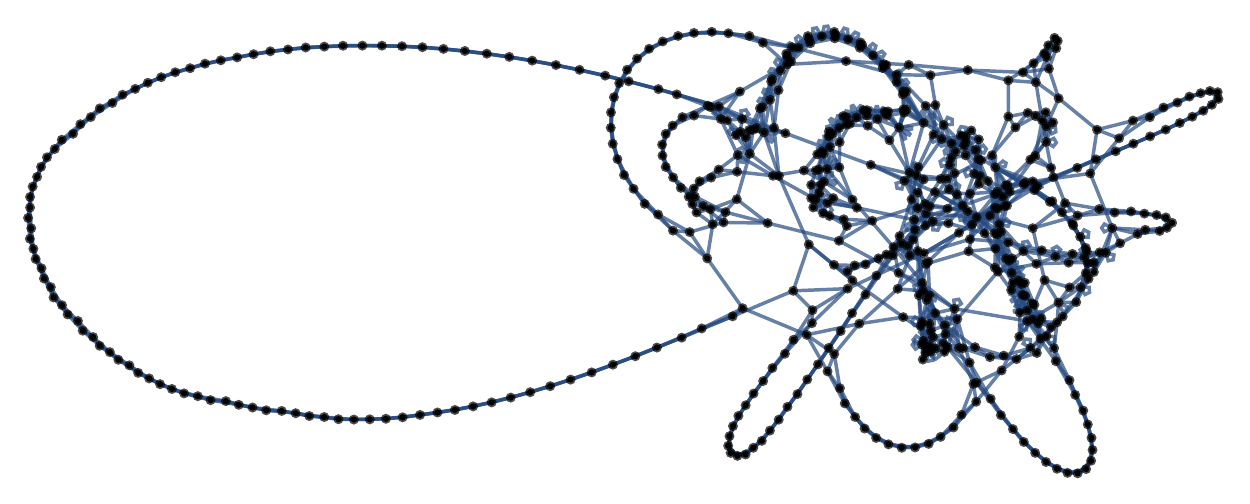}};
\end{tikzpicture}
\caption{The Recaman sequence (A005132) with duplicates removed.}
\end{figure}
\end{center}

\subsection{A Sequence of Leroy Quet} We learned about this sequence (A127202) from a survey article of Sloane \cite{notices}. It is given by $a_1 =1, a_2 = 2$ and then by
picking $a_n$ as the smallest positive integer that has not yet occurred in the sequence satisfying the property that $\gcd(a_n, a_{n-1}) \neq \gcd(a_{n-1}, a_{n-2})$. The sequence begins
$$ 1, 2, 4, 3, 6, 5, 10, 7, 14, 8, 9, 12, 11, 22, \dots$$
Quet conjectured that this sequence is a permutation of the integers, a statement proven by Sloane.  We obtain a graph that clearly exhibits some nontrivial structure -- it is not clear to us why or whether it will continue in this fashion (see Fig. \ref{fig:quet}).

\subsection{Zabolotskiy's sequence} Zabolotskiy's sequence (A281488) is given by a simple rule: $a_1 = 1, a_2 = -1$ and, for $n \geq 3$,
$$ a_n  = - \sum_{1 \leq d \leq n-2 \atop d| n-2 } a_d.$$
After removing the duplicates, we arrive at a very strange type of graph: it clearly exhibits a degree of regularity but it is very difficult to say why this would be the case or whether it will continue to exhibit such structure once the graphs get larger, see Fig. \ref{fig:quet}.

\begin{center}
\begin{figure}[h!]
\begin{tikzpicture}
\node at (0,0) {\includegraphics[width=0.45\textwidth]{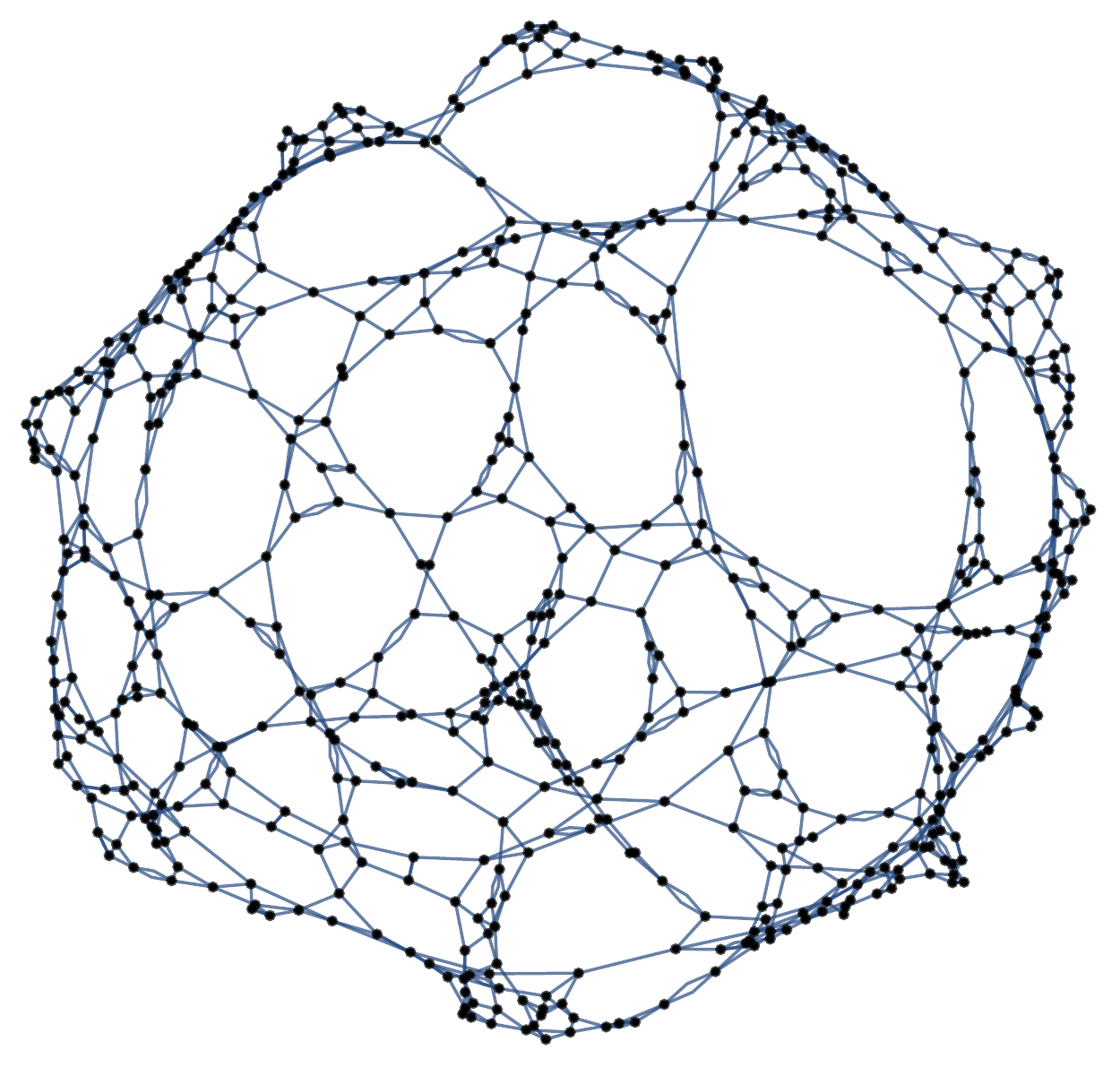}};
\node at (6,0) {\includegraphics[width=0.45\textwidth]{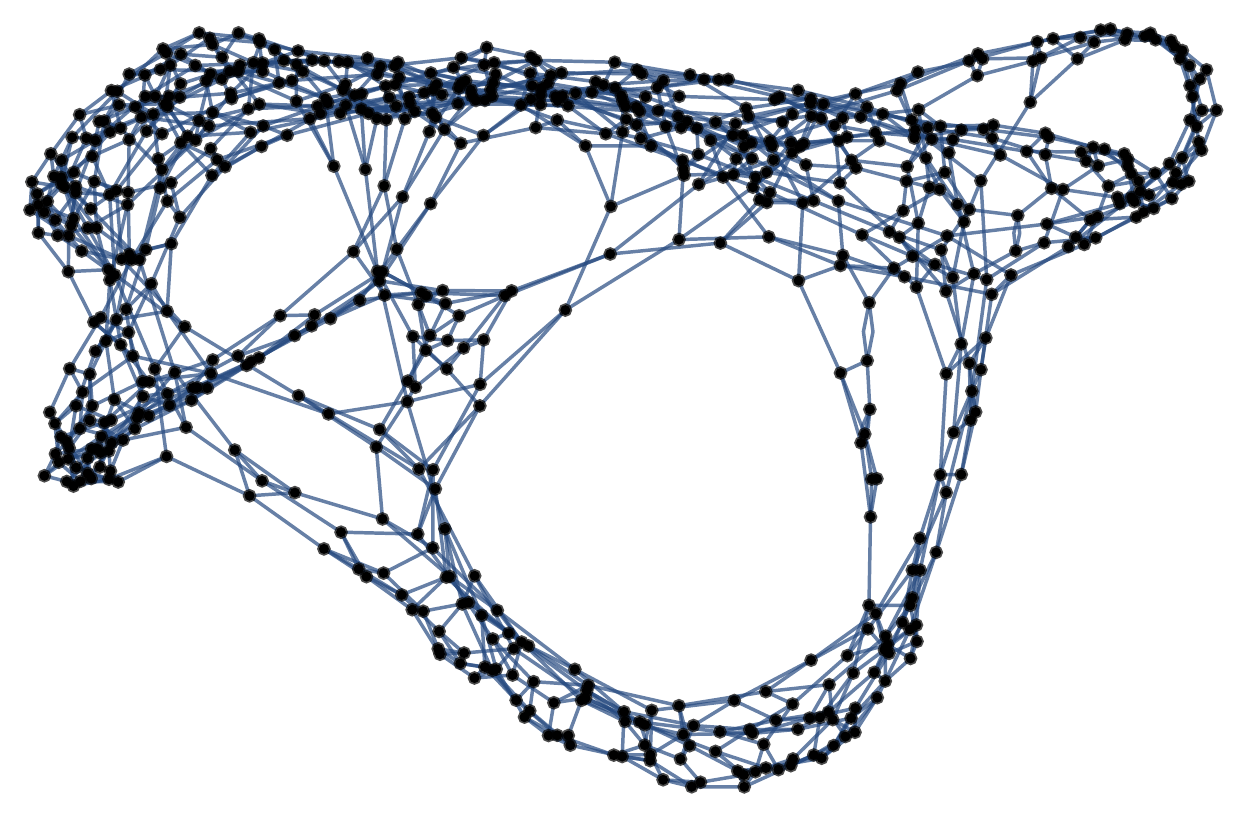}};
\end{tikzpicture}
\caption{Quet's sequence (left) and Zabolotskiy's sequence (right).}
\label{fig:quet}
\end{figure}
\end{center}

\subsection{Gray Encoding} Before defining the sequence, we define the `derivative' of an integer: write $n$ in binary and then replace each adjacent pair of bits by their sum mod 2. For example, we have
$$ 18  = 10011_2$$ 
and thus the derivative is given by the digits $(1+0), (0+0), (0+1), (1+1)$ (all taken mod 2), which correspond to $1010_2 = 10$. Thus, in this language, the derivative of 18 is 10. The `derivative sequence' is A038554 in the OEIS. Having introduced this notion, we can now define A006068 as the permutation of the integers in such an order that their derivatives are lexicographically ordered. 
To be more precise, the sequence of derivatives (A038554) starts, written in binary, as
$$ -,-,1,0,10,11,01,00,100,101,111,110,010,011,001,000, \dots$$
We see that the proper lexicographic ordering is $0,1,3,2,7,6,4,5,15,\dots$ which is exactly the sequence under consideration (see Fig. \ref{fig:gray}).

\begin{center}
\begin{figure}[h!]
\begin{tikzpicture}
\node at (0,0) {\includegraphics[width=0.4\textwidth]{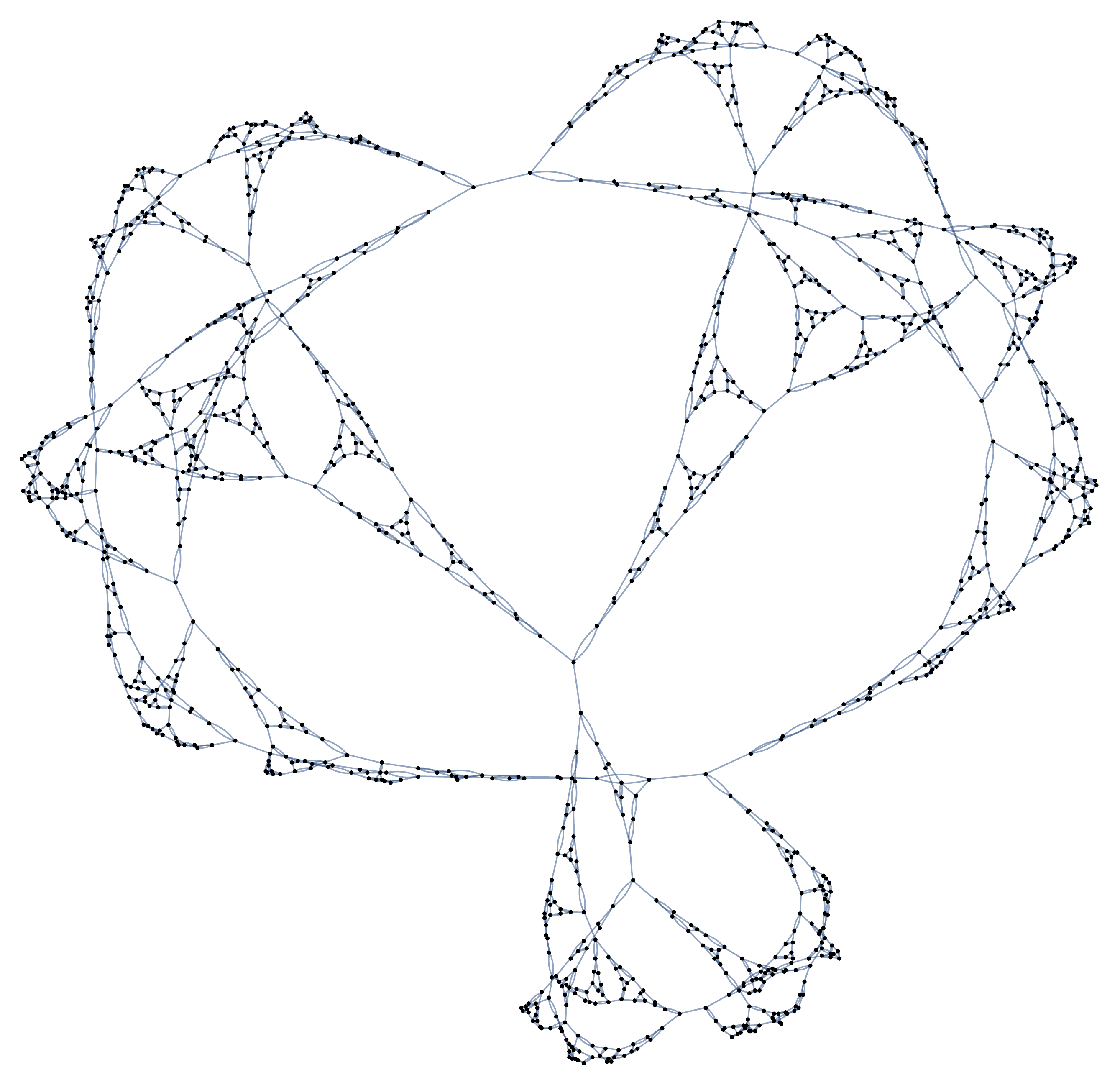}};
\node at (6,0) {\includegraphics[width=0.4\textwidth]{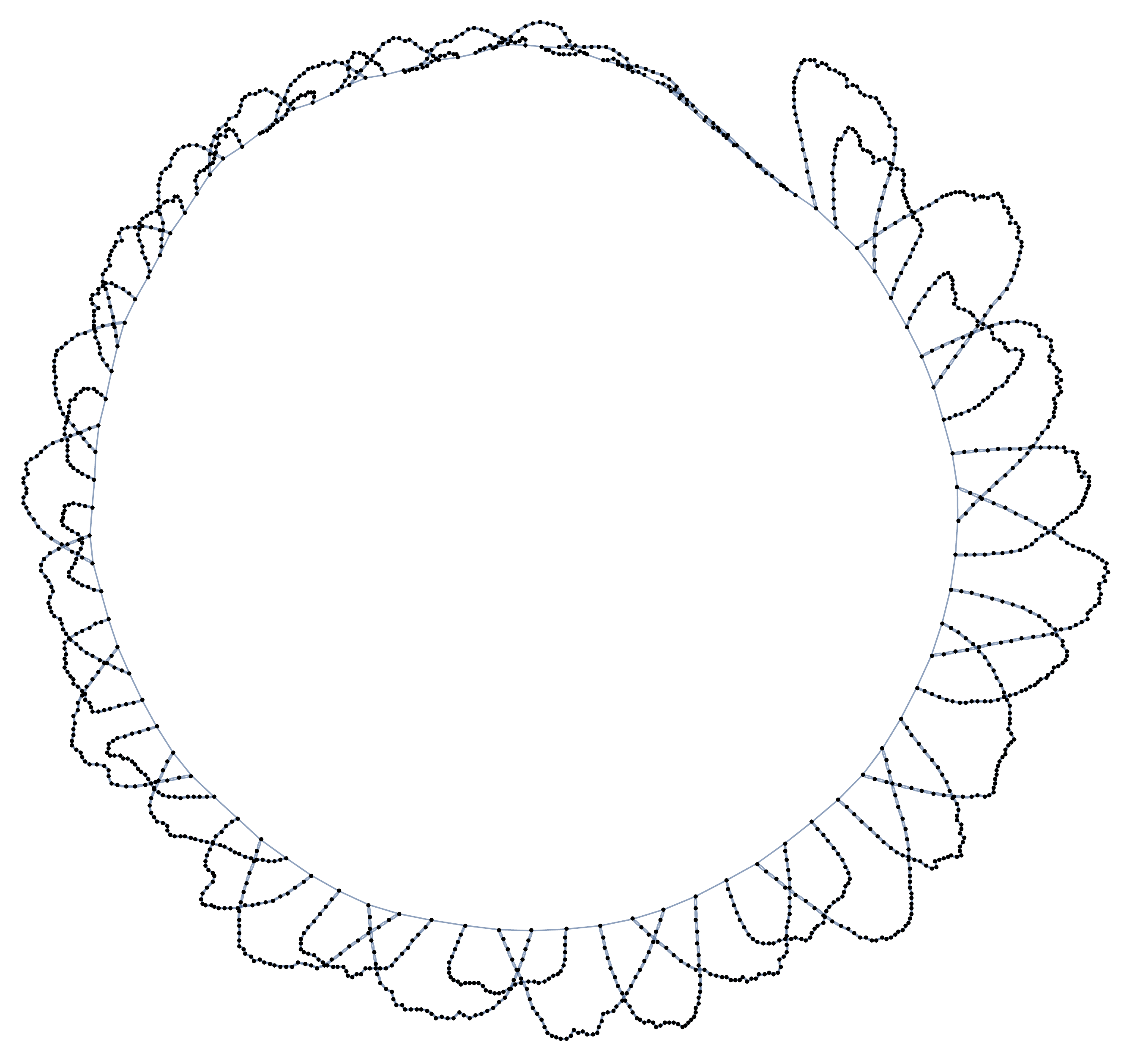}};
\end{tikzpicture}
\caption{Graphs for the Gray encoding (left) and A140589 (right).}
\label{fig:gray}
\end{figure}
\end{center}

\subsection{Two Sums of Powers of 2} This fairly simple sequence is given by defining
$$ a_{k,n} = (-2)^k + 2^n, \; 0\le n\le k$$
which is then read by rows leading to the sequence
$$2, -1, 0, 5, 6, 8, -7, -6, -4, 0, 17, 18, 20, 24, 32, \dots$$
This example is simple enough that we can analyze it completely (see Fig. \ref{fig:gray}). Recall the in order to construct the graph described in the section \ref{subsection:Graph}, we need the sequence $a_{k, n}$ to be a sequence of distinct element. It is easy to see that there are repeated entries in the sequence, in particular, $a_{k, n}=0$ for $k$ odd and $n=k.$ But these are the only duplicates in the sequence. Our next lemma makes this precise.
\begin{lemma}
\label{Lemma:Injectivity_A140589} Let $a_{k, n}$ be as defined in \S 2.13. For $0\le n\le k_1$ and $0\le m\le k_2,$ such that $(k_1, n)\neq (k_2, m),$ we have $a_{k_1, n}=a_{k_2, m}$ if and only if $k_1, k_2$ are odd and $n=k_1$ and $m=k_2.$ 
\end{lemma}
\begin{proof}
If part is clearly trivial. Assume that $(k_1, n)\neq (k_2, m).$ If $k$ is odd then $a_{k, j}\le 0$ for all $j$ and if $k$ is even then $a_{k, j}>0$ for all $j.$ It follows that for $a_{k_1, n}=a_{k_2, m}$ to hold, $k_1, k_2$ must have the same parity. 
Assume that $k_1<k_2$ are even. It follows that $a_{k_1, n}\le 2^{k_1+1}<2^{k_2}<a_{k_2, m}.$ Now assume that $k_1<k_2$ are odd and $n<k_1, m<k_2.$ In this case $a_{k_2, m} \in [-2^{k_2}+1, -2^{k_2-1}],$ while $a_{k_1, n}\in [-2^{k_1}+1, -2^{k_1-1}].$ The two intervals are disjoint and therefore $a_{k_1, n}\neq a_{k_2, m}.$
\end{proof}

We consider the subsequence of the above sequence obtained after removing all zeroes from the sequence. We now describe the graph (see Fig. \ref{fig:gray}) obtained from this subsequence. Note that for a fixed $k,$ the sequence $a_{k, m}$ is monotonically increasing.  From the proof of lemma \ref{Lemma:Injectivity_A140589} it is clear that this ordering is preserved even after sorting the whole sequence in increasing order. This suggests that for each $k,$ we have a chain of length $k$ (or $k-1$ if $k$ is odd) corresponding to the subsequence $a_{k, n}.$ This chains are joined end-to-end via single edge to form a full circle. The end-nodes of each chain, that is, the nodes corresponding to $a_{k, 0}$ and $a_{k, k-1_{k \text{ odd}}},$ are joined with the end-nodes of the second next chain. That is, $a_{k, k}$ is connected to $a_{k+2, 0}$ if $k$ is even. If $k$ is odd, $a_{k, k-1}$ is connected to $a_{k+2, k+1}.$ 

To make the above discussion precise, we describe the adjacency matrix of the graph obtained from this sequence in the following lemma. We take the vertex set of this graph  to be $V_N=\{(k, n): 0\le n\le k-1_{n \text{odd}}, k\le N\}$ for some $N.$
\begin{lemma}
\label{Lemma:Adjaceny_A140589}
Fix $N\in \mathbb{N}$ and let $V_N$ be as above. Let $A$ be the matrix with rows and columns indexed by $V$ with entries give as follows:
\[A((k, m), (l, n)) = \begin{cases} 2, & l=k, |m-n|=1\\
1, & l=k+1, m=k-1_{k \text{odd}}, n=0\\
1, & l=k+2, m=k-1_{k \text{odd}},  n=(k+1)1_{ k \text{odd}}\\
1, & (k, n)=(1, 0), (l, m)=(N, N-1_{N \text{odd}})\\
0, & \text{otherwise}
\end{cases}.\]  
Then, $A_N$ is the adjacency matrix of the graph described in \ref{subsection:Graph}
corresponding to the above sequence (after removing zeroes).
\end{lemma}
\begin{proof}
The proof follows from the proof of Lemma \ref{Lemma:Injectivity_A140589}.
\end{proof}

\subsection{A Digit Reversal Sequence} This sequence (A059893) is a simple permutation of the integers: if $n$ has expansion
$1 a_1 \dots a_{\ell}$ in base 2, then $a_n$ is given by the integer $1 a_{\ell} \dots a_1$. It is easy to see that this is a permutation. One would
perhaps also expect that it is somehow connected to the van der Corput sequence and it does, visually, look somewhat similar: we see several
small manifolds that are connected via a line.

\begin{center}
\begin{figure}[h!]
\begin{tikzpicture}
\node at (0,0) {\includegraphics[width=0.4\textwidth]{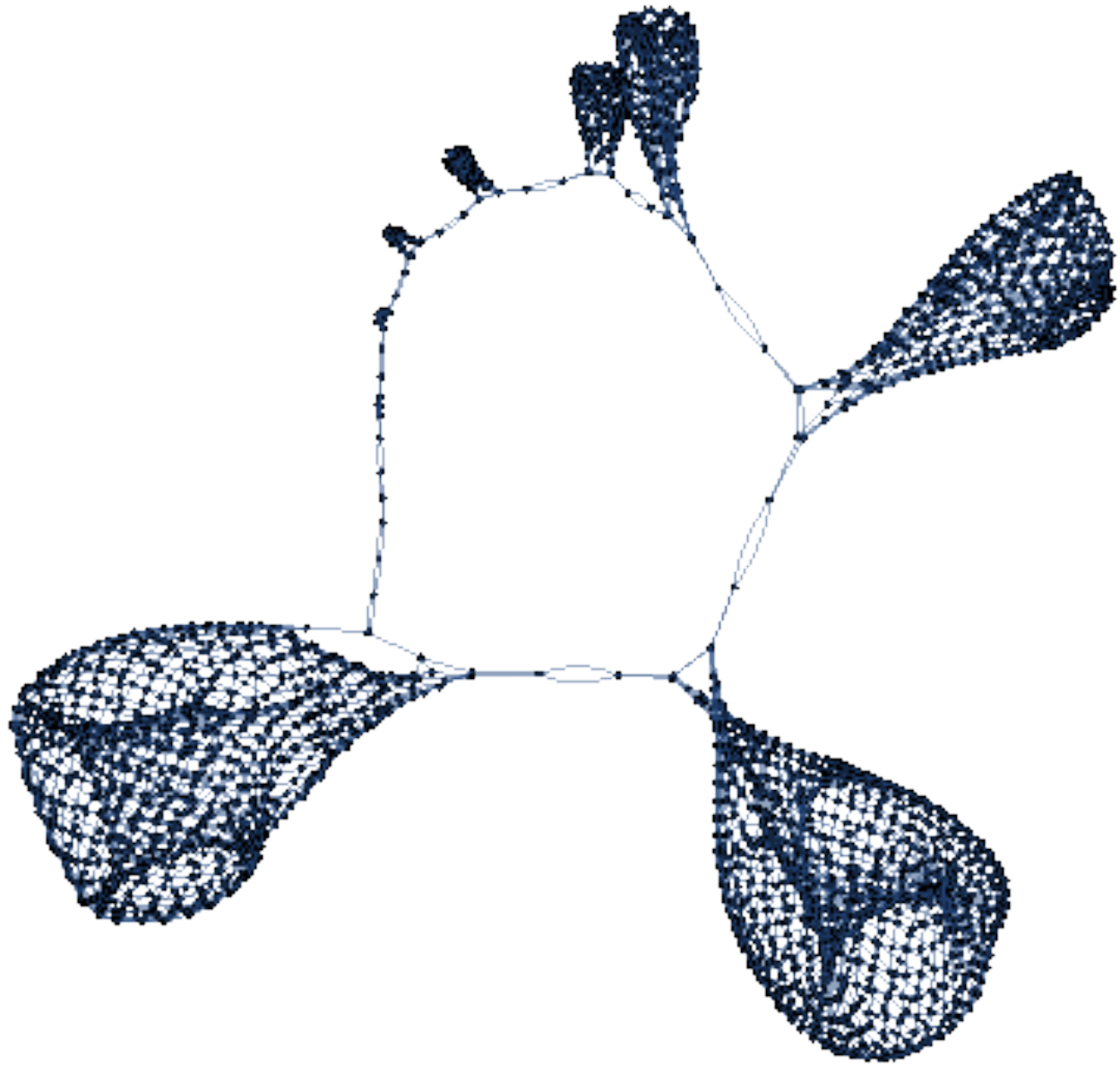}};
\node at (6,0) {\includegraphics[width=0.5\textwidth]{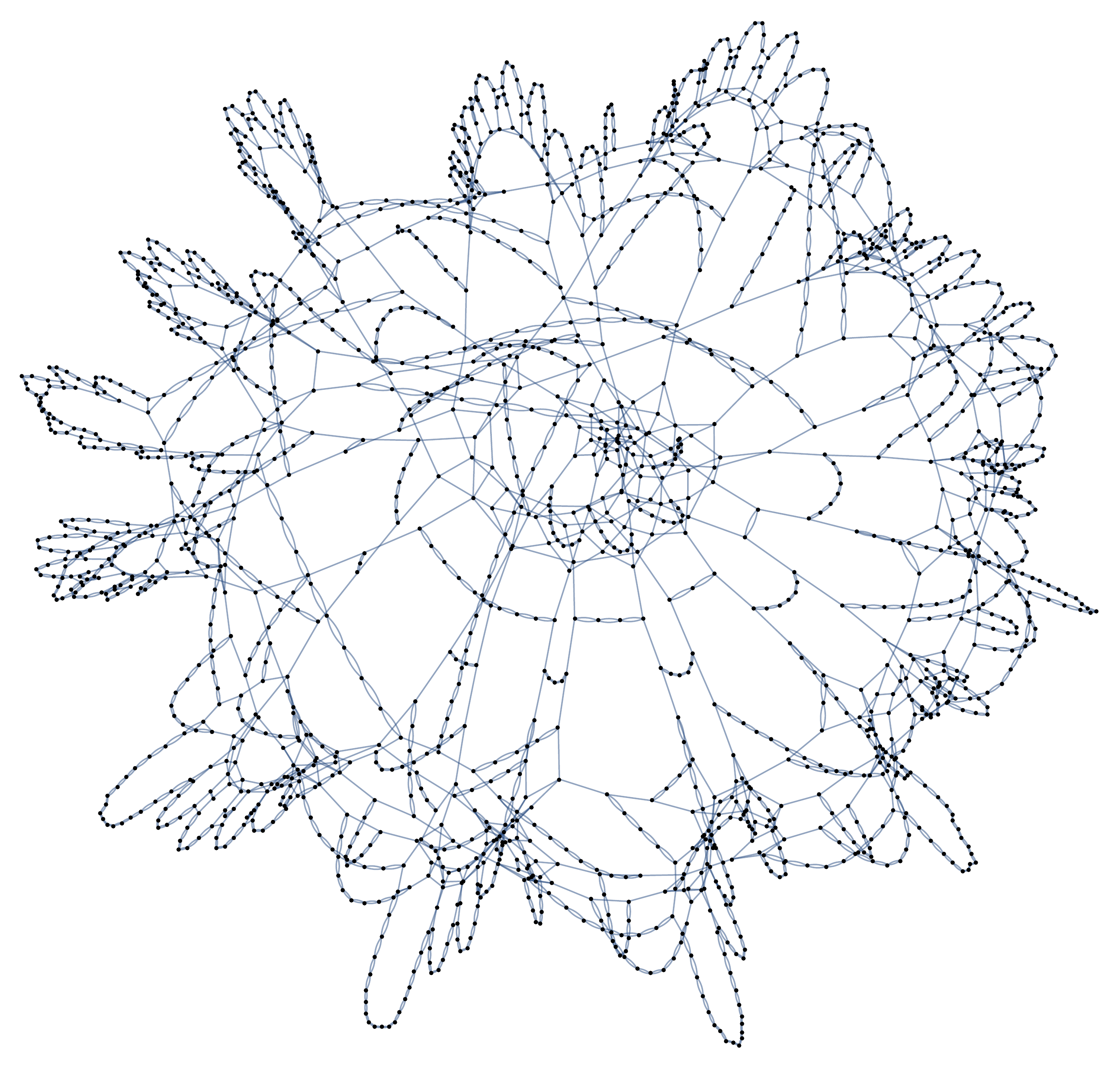}};
\end{tikzpicture}
\caption{Graphs for A059893 (left) and A347520 (right).}
\label{fig:concat}
\end{figure}
\end{center}

\subsection{A Digit Concatenation Sequence} This example is obtained by considering the sequence A053392 and then removing all duplicates resulting in A347520. A053392 is easy to define: we write $n$ in decimal digits, consider each pair of consecutive digits and replace them by their sum. The integer 9327, for example, would lead to $(9+3).(3+2).(2+7) = 12.5.9 = 1259$ and thus $a_{9327} = 1259$. Removing duplicates, we end up with the sequence under consideration -- it results in a striking and highly regular graph, see Fig. \ref{fig:concat}.

\subsection{A Construction of Emeric Deutsch}  A057163 starts 
$$ 0, 1, 3, 2, 8, 7, 6, 5, 4, 22, 21, 20, 18, 17, \dots$$
This sequence is related to a combinatorial construction due to Emeric Deutsch \cite{deutsch} regarding Dyck paths. The sequence itself is due to A. Karttunen.

\begin{center}
\begin{figure}[h!]
\begin{tikzpicture}
\node at (0,0) {\includegraphics[width=0.85\textwidth]{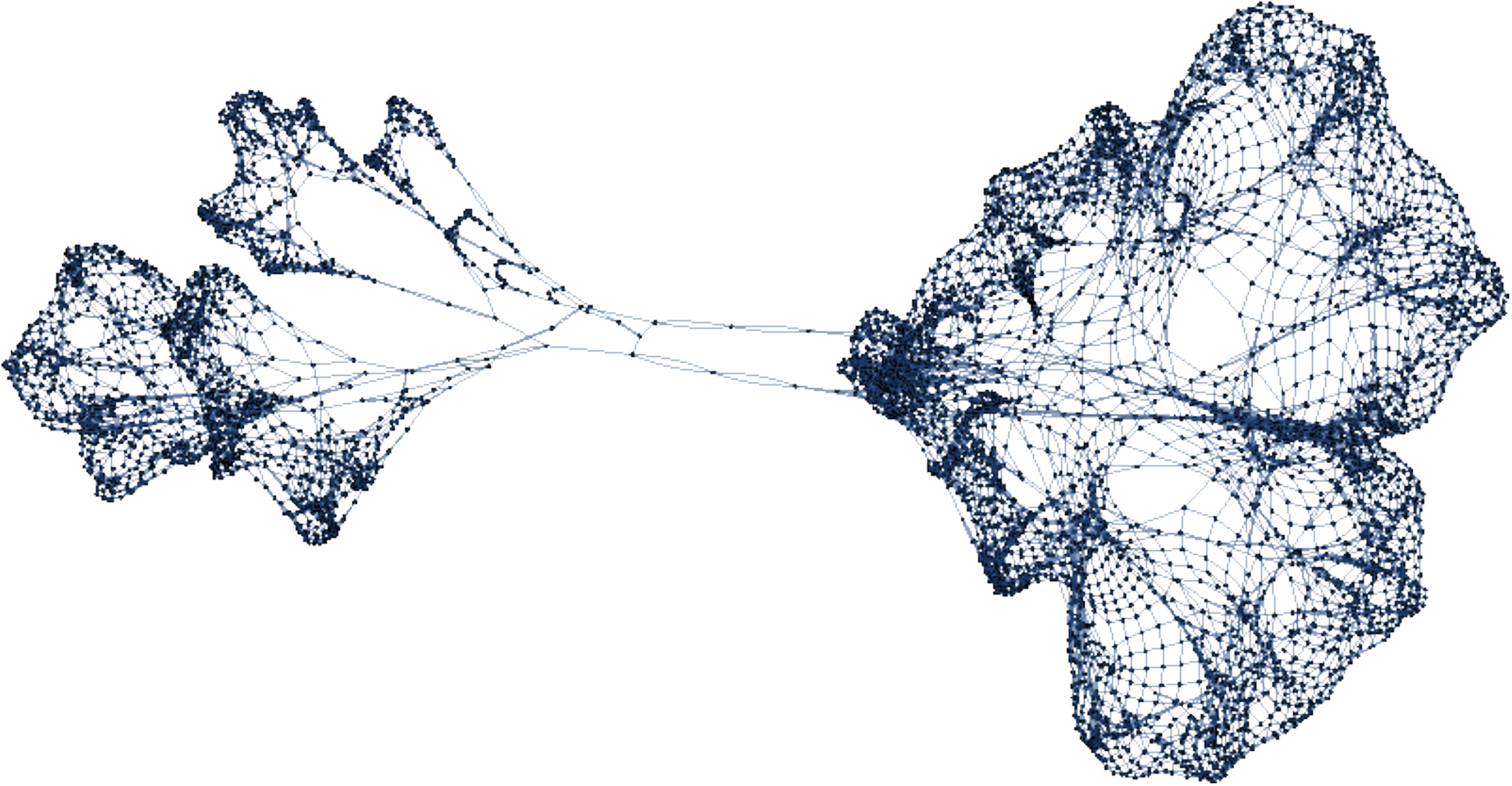}};
\end{tikzpicture}
\caption{Graph for the sequence (A057163).}
\end{figure}
\end{center}

\section{Some Universal Principles}
It is clear that the structure of these graphs is intimately connected to the sequence. As such, we would not expect there to be many universal results. However, while investigating the phenomenon, there are two types of structures that we have encountered several times: we call them the `comets' and the `spirals'.

\subsection{The Comet} The comet is perhaps the shape that we encountered most frequently: we refer to Fig. \ref{fig:comet} for two fairly typical examples. To the best of our knowledge, comets tend to arise when dealing with a sequence $x_n$ with (i) very few long monotonic consecutive subsequences and (ii) monotonic subsequences whose indices do not obey any apparent pattern. A way to obtain comet-like sequences experimentally is as follows: let $X \sim \mathcal{U}(0,1)$ be a random variable that is uniformly distributed on $[0,1]$ and define the (random) sequence
$$ a_n = n + c \cdot X \cdot n \qquad \mbox{where}~c \in \mathbb{R}.$$

\begin{center}
\begin{figure}[h!]
\begin{tikzpicture}
\node at (0,0) {\includegraphics[width=0.4\textwidth]{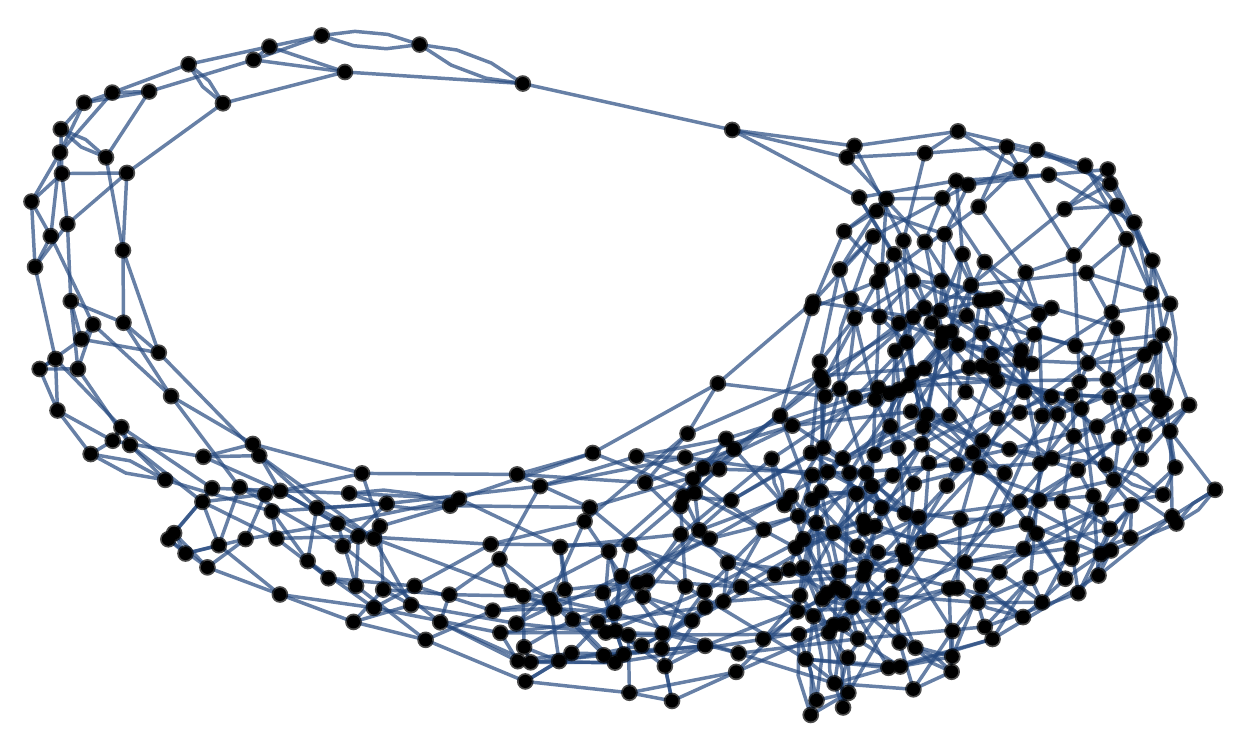}};
\node at (6,0) {\includegraphics[width=0.4\textwidth]{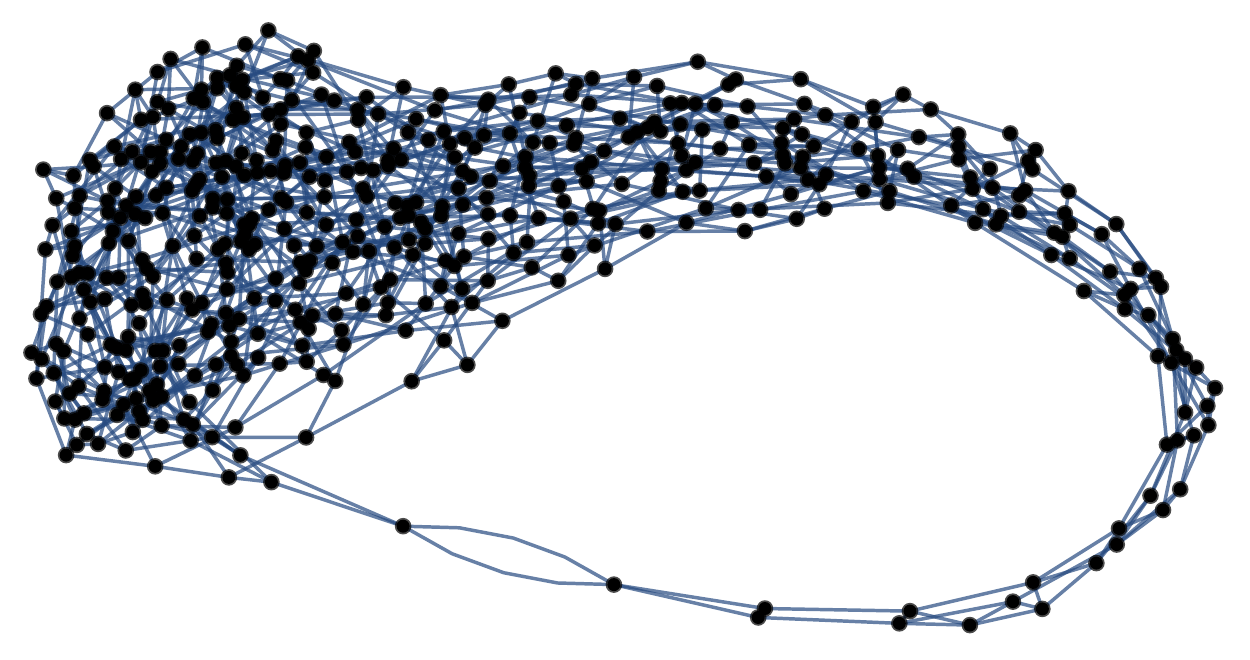}};
\end{tikzpicture}
\caption{A061020 (left) and A280864 (right) are `comets'.}
\label{fig:comet}
\end{figure}
\end{center}

We see that such sequences are growing, on average, but due to the randomness it is difficult to say whether any given element is larger than any other element with a nearby index.  Many integer sequences that lead to comets seem to of such a flavor. Note that our Graph construction is only determined by the ordering of elements but not by their size and thus, naturally, one could also deal with sequences of the form $a_n = e^{n} + c \cdot X \cdot e^{n}$. The constant $c$ governs the `thickness' of the comet.

\subsection{Spirals}
The second type of frequently encountered object looks a bit like a spiral (see Fig. \ref{fig:spiral}). We do not have a clear characterization but it does seem like such sequences arises from having consecutive segments of increasing subsequences. A natural example (also shown in Fig. \ref{fig:spiral}) is the sequences of binomial coefficients in the Pascal triangle read by row: since we remove all duplicates, we only ever read the first half of each row leading to a natural increasing subsequnce.

\begin{center}
\begin{figure}[h!]
\begin{tikzpicture}
\node at (0,0) {\includegraphics[width=0.4\textwidth]{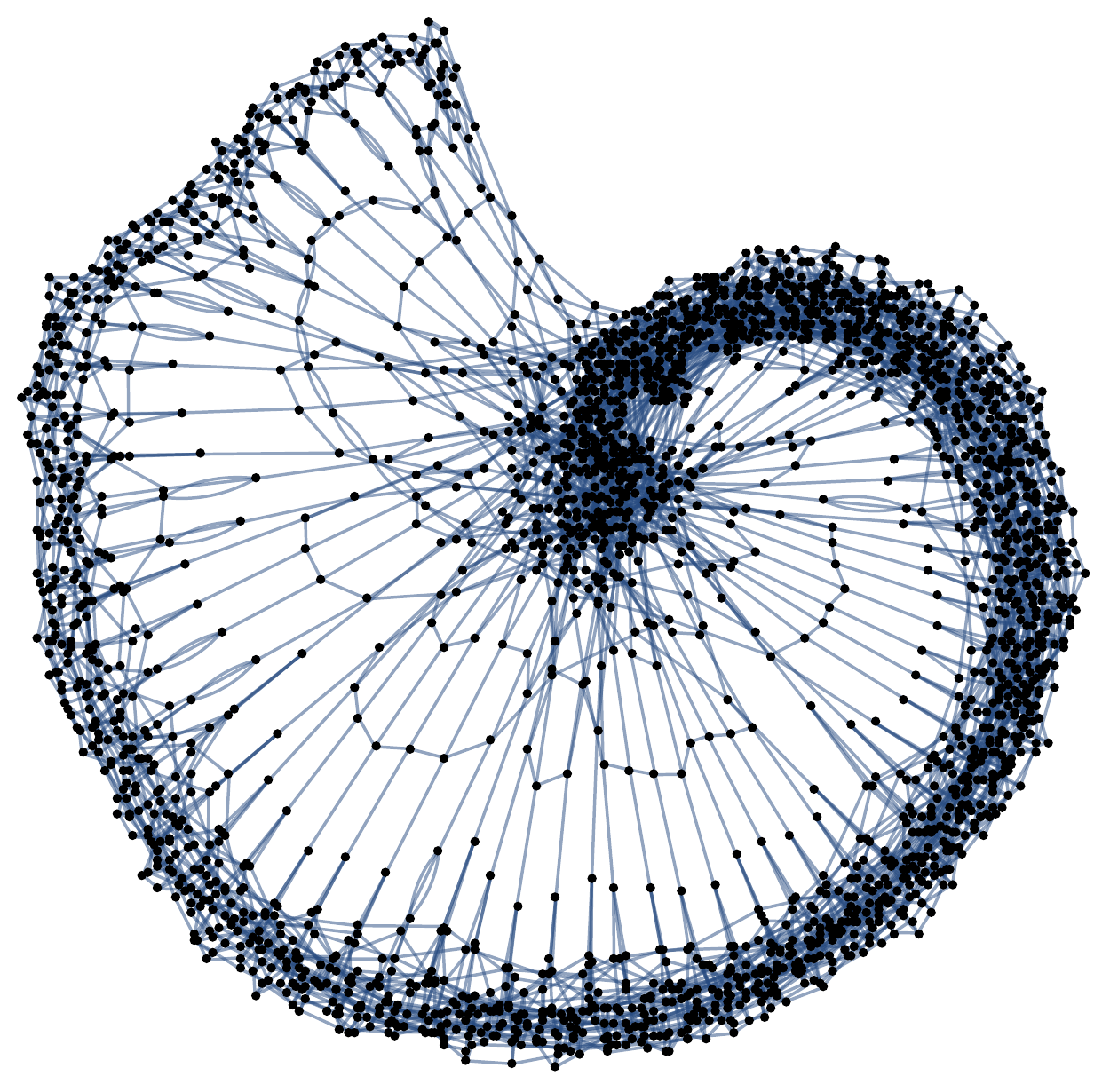}};
\node at (6,0) {\includegraphics[width=0.4\textwidth]{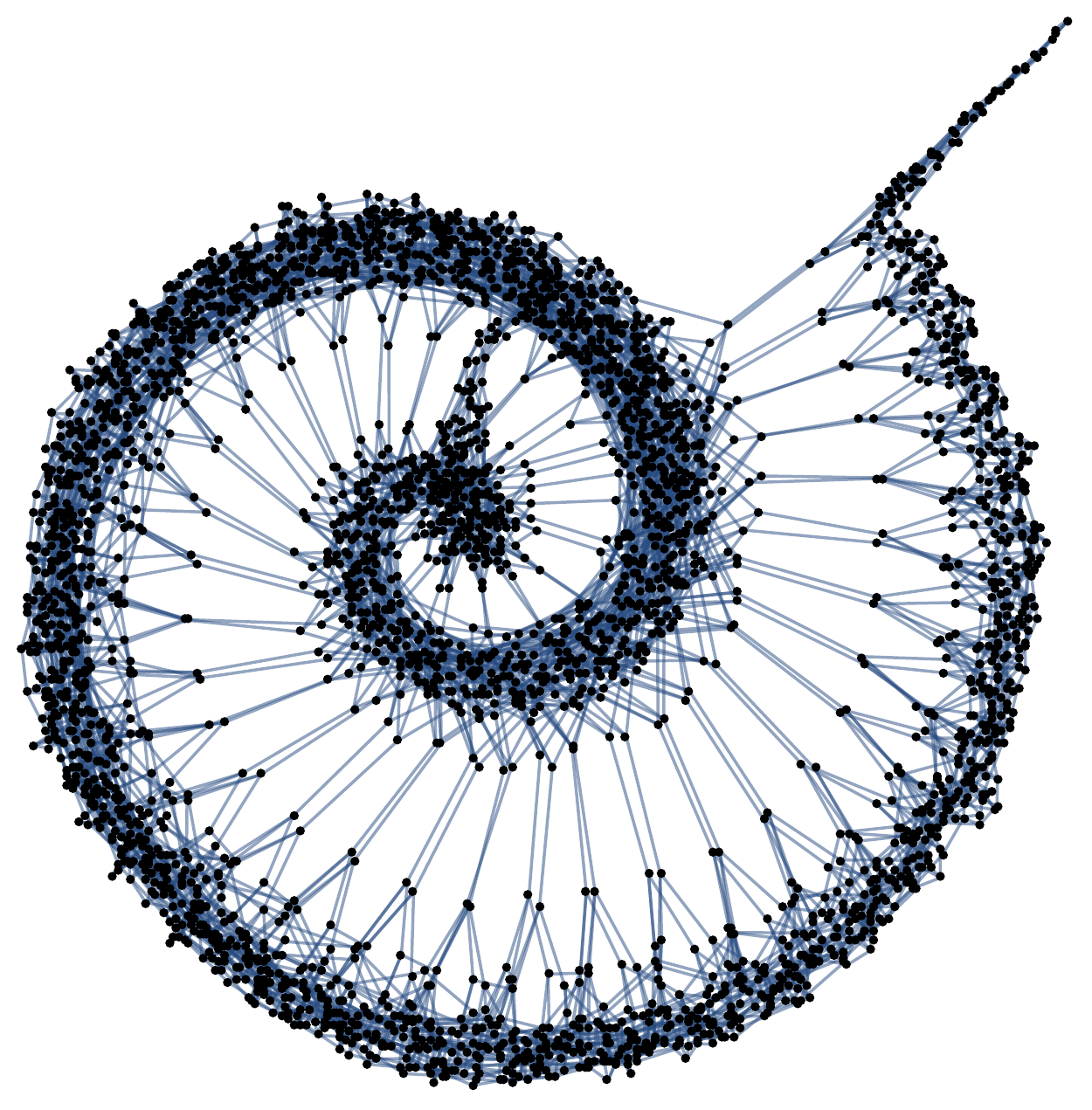}};
\end{tikzpicture}
\caption{Left: A014631 (rows of Pascal's triangle with duplicates removed). Right:  Triangle of nonzero coefficients of Hermite polynomials $H_n(x)$ in increasing powers of $x$ (A059343).}
\label{fig:spiral}
\end{figure}
\end{center}

While trying to come up with a way of `artificially' creating sequences that lead to spirals, we found the following model. Let $f:\mathbb{N} \rightarrow \mathbb{R}_{\geq 0}$ be a slowly increasing sequence, maybe $f(n) = (\log{(n+1)})^3$ or $f(n) = n^{1/10}$. We then define
$$ a(n, k) = \exp\left(k \cdot (f(n) - f(k)) \right)$$
and look at the triangle
\begin{align*}
a(1,1)\\
a(2,1) \quad a(2,2)\\
a(3,1) \quad a(3,2) \quad a(3,3)\\
\dots
\end{align*}
which we then read by rows. After removing the duplicates (note that $a(n,n) = 1$), we end up with sequences that seem to lead to spirals in a natural way. This construction is inspired by binomial coefficient asymptotics, the full extent of the phenomenon is not currently understood.

\subsection{Analyzing graphs.} We recall the definition of our graphs: we define the graph associated to $n$ distinct real numbers $a_1, \dots, a_n$ by adding edges between $a_i$ and $a_{i+1}$ as well as $a_1$ and $a_n$. Then, defining the (unique) permutation 
$\pi: \left\{1,2,\dots,n\right\} \rightarrow \left\{1,2,\dots, n\right\}$ so that
$$a_{\pi(1)} < a_{\pi(2)} < \dots < a_{\pi(n)},$$
we connect $a_{\pi(i)}$ to $a_{\pi(i+1)}$ and conclude by also connecting to $a_{\pi(n)}$ to $a_{\pi(1)}$. One naturally relevant
question is how one would go about discovering `structure' in a graph (or what this even means). We used the eigenvalue as a spectral quantity \S 1.3
in combination with Mathematica's \textsc{Graph} command. For the sake of reproducibility, we quickly explain how this was done using the Kronecker sequence
$a_n = \left\{ n \sqrt{2} \right\}$ as an example. We first create a list containing the first 200 elements of the sequence
\begin{verbatim}
a = Table[N[FractionalPart[Sqrt[2]*i]], {i, 1, 200}];
\end{verbatim}
Having defined the sequence, we define a graph via
\begin{verbatim}

b = Sort[a]; G = Graph[Join[Table[ b[[i]] \[UndirectedEdge] b[[i + 1]], 
{i, 1, Length[b] - 1}],  {a[[1]] \[UndirectedEdge] a[[-1]]}, 
{b[[1]] \[UndirectedEdge] b[[-1]]}, Table[a[[i]] \[UndirectedEdge] 
a[[i + 1]], {i, 1, Length[a] - 1}]]];
    \end{verbatim}
Finally, it remains to see whether there is any structure in the graph. The second largest eigenvalue of the adjacency matrix is given by
\begin{verbatim}
Eigenvalues[N[AdjacencyMatrix[G]]][[2]]
\end{verbatim}
and returns the value $-3.959$ whose absolute value is indeed very close to 4 (a sufficient condition for a lot of structure being present). The
eigenvalue strongly suggests the presence of structure, we would therefore like to have a look at the graph.
Mathematica has the ability to produce a variety of graph embeddings.  Fig. 17 is easy to produce using the code
\begin{verbatim}
Table[Graph[G, GraphLayout -> l,  PlotLabel -> l], {l, 
{"CircularEmbedding", "SpiralEmbedding", "SpringEmbedding", 
"SpringElectricalEmbedding", "HighDimensionalEmbedding",
"BalloonEmbedding", "RadialEmbedding", "GravityEmbedding", 
"SpectralEmbedding"}}]
\end{verbatim}
\begin{center}
\begin{figure}[h!]
\begin{tikzpicture}
\node at (0,0) {\includegraphics[width=0.8\textwidth]{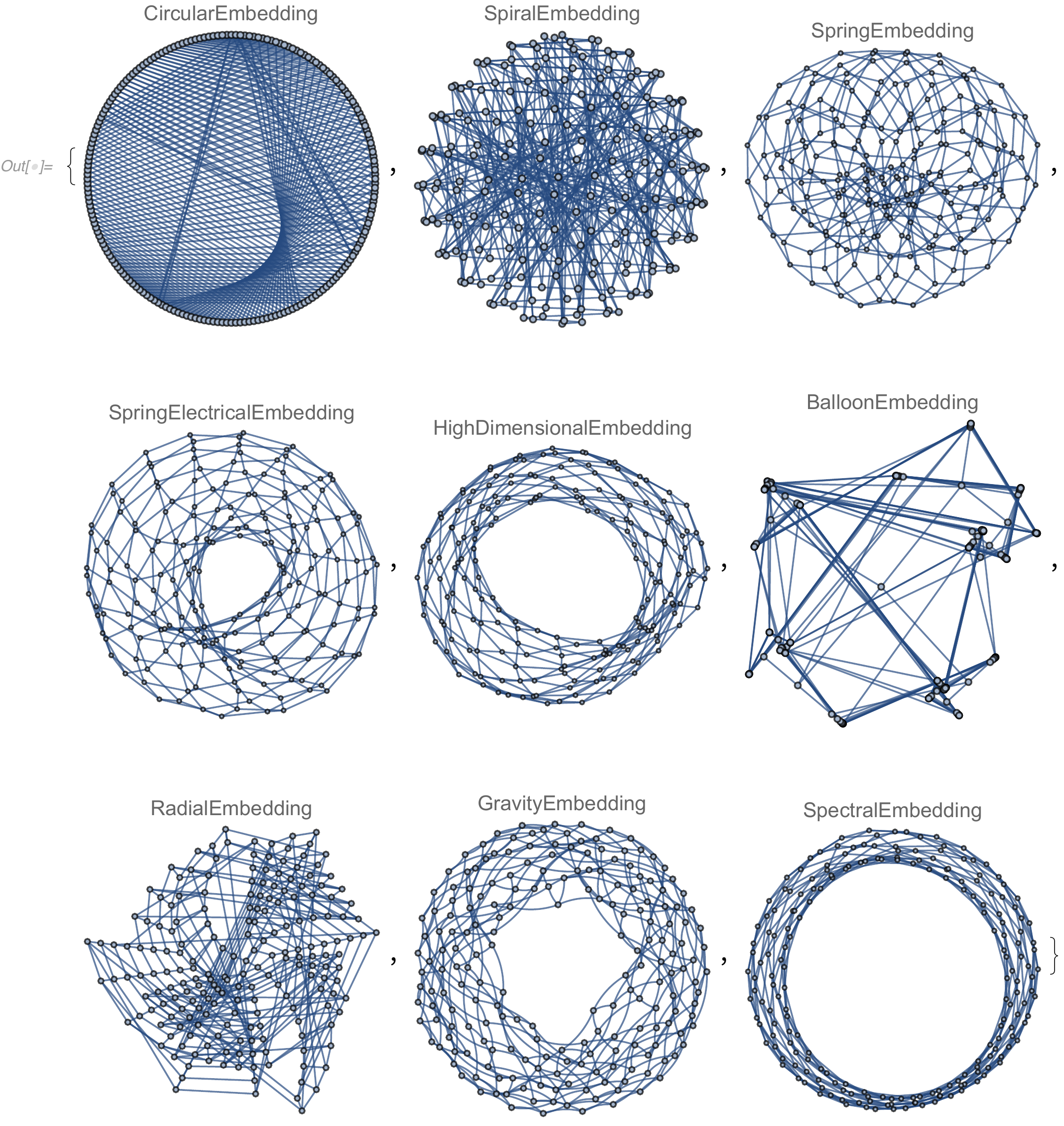}};
\end{tikzpicture}
\caption{Embeddings produced for this particular example.}
\label{fig:spiral}
\end{figure}
\end{center}

We see that some of the embeddings are better suited at revealing the underlying structure than others. However, what
is clearly being shown is that the graph $G$ is indeed highly structured and regular. In practice, we have found that 
spectral embedding methods tend to be very flexible. \\

\textbf{Acknowledgment.} We are grateful to Neil Sloane for helpful discussions.

\end{document}